\documentclass[preprint,sort&compress,12pt]{elsarticle}
\usepackage{bbm}
\usepackage{amstext}

\usepackage{amsmath}
\usepackage{amssymb}
\usepackage{mathrsfs}
\usepackage{bm}
\usepackage{pstricks}
\usepackage{pst-coil}
\usepackage{pst-3d}
\usepackage{amsfonts}
\usepackage{amsthm}
\usepackage{CJK}
\usepackage{xcolor}
\allowdisplaybreaks
\def\journal{\quad }   
\def\ftimes{ 
 }

\newcommand{\mm}{\mathrm}

\newcommand{\be}{\begin{equation}}
\newcommand{\bea}{\begin{equation}\begin{aligned}}
\newcommand{\beas}{\begin{equation*}\begin{aligned}}
\newcommand{\eeas}{\end{aligned}\end{equation*}}
\newcommand{\eea}{\end{aligned}\end{equation}}
\newcommand{\ee}{\end{equation}}

\begin{document}
\begin{CJK*}{GBK}{song}
\begin{frontmatter}
\title{Asymptotical Behavior of Global Solutions of the Navier--Stokes--Korteweg Equations with Respect to\\
 Capillarity Number at Infinity}
\author[FJ,sJ]{Fei Jiang}
\ead{jiangfei0591@163.com}
\author[FJ]{Pengfei Li\corref{cor1}}
 \ead{pfliyou@163.com}
\author[FsJ]{Jiawei Wang}
\ead{wangjiawei@amss.ac.cn}
\cortext[cor1]{Corresponding author.} 
\address[FJ]{School of Mathematics and Statistics, Fuzhou University, Fuzhou, 350108, China.}
\address[sJ]{Key Laboratory of Operations Research and Control of Universities in Fujian, Fuzhou 350108, China.}
\address[FsJ]{Hua Loo-Keng Center for Mathematical Sciences, Chinese Academy of Sciences,
Beijing 100190, China. }
\begin{abstract}
Vanishing capillarity in the Navier--Stokes--Korteweg (NSK) equations has been widely investigated, in particular, it is well-known that the NSK equations converge to the Navier--Stokes (NS) equations by vanishing capillarity number. To our best knowledge, this paper first investigates the behavior of large capillary number, denoted by $\kappa^2$, for the global(-in-time) strong solutions with small initial perturbations of the three-dimensional (3D) NSK equations in a slab domain with Navier(-slip) boundary condition. Under the well-prepared initial data, we can construct a family of global strong solutions of the 3D incompressible NSK equations with respect to $\kappa>0$, where the solutions converge to a unique solution of  2D incompressible NS-like equations as $\kappa$ goes to infinity.
\end{abstract}
\end{frontmatter}
\newtheorem{thm}{Theorem}[section]
\newtheorem{lem}{Lemma}[section]
\newtheorem{pro}{Proposition}[section]
\newtheorem{concl}{Conclusion}[section]
\newtheorem{cor}{Corollary}[section]
\newproof{pf}{Proof}
\newdefinition{rem}{Remark}[section]
\newtheorem{definition}{Definition}[section]

\section{Introduction}\label{introud}
\numberwithin{equation}{section}
A classical model to  describe the dynamics of an inhomogeneous
 incompressible fluid endowed with internal capillarity  (in the diffuse interface setting) is the following  {general system of} incompressible Navier--Stokes--Korteweg (NSK) equations:
\begin{align}\label{1.1}
\begin{cases}
\rho_t  +    \mathrm{div}(\rho v) = 0, \\
\rho (v_t +  v\cdot \nabla v) - \mathrm{div}(\mu (\rho )\mathbb{D}v)  + \nabla  {P}  = \mathrm{div}  {K}  ,\\
\mm{div}  v=0,
\end{cases}
\end{align}
where ${\rho (x,t)} \in \mathbb{R}^{+}$, $v(x,t) \in {\mathbb{R}^3}$ and $P(x,t)$ denote the density, velocity and kinetic pressure of the fluid {resp.} at the spacial position $x \in {\mathbb{R}^3}$ for time $t\in\mathbb{R}^+_0:=[0, {+\infty} )$. The differential operator $\mathbb{D}$ is defined by $\mathbb{D}v = \nabla v+\nabla v^{\top}$, where the superscript $\top$ represents the transposition. The shear viscosity function $\mu$, the capillarity function $\tilde{\kappa}$ are known smooth functions $\mathbb{R}^+ \to \mathbb{R}$, and satisfy $\mu>0$, $\tilde{\kappa}>0$. The general capillary tensor is written as
\begin{align*}
K=\left( {\rho \mm{div}(\tilde{\kappa}(\rho )\nabla \rho ) + \left( {\tilde{\kappa}(\rho ) - \rho\tilde{\kappa}'(\rho )} \right){{\left| {\nabla \rho } \right|}^2}} /2\right)\mathbb{I} -  {\tilde{\kappa}(\rho )\nabla \rho  \otimes \nabla \rho }, 
\end{align*}
where $\mathbb{I}$ is the identity matrix. To conveniently investigate the asymptotical behavior of capillarity number, we considers that \emph{$\mu$ and $\tilde{\kappa}$ are positive constants} \cite{BDG2008}, thus it holds that
\begin{align}
\label{202501232119}
 \mathrm{div}(\mu (\rho )\mathbb{D}v) =\mu\Delta v\mbox{ and } \mathrm{div}K =\kappa^2 \rho\nabla \Delta \rho,
\end{align}
where we have defined that $\kappa:=\sqrt{\tilde{\kappa}}>0$. We call 
the parameters $\mu$ and $\kappa^2$ the viscosity coefficient and the capillarity number, resp..

In classical hydrodynamics, the interface between two immiscible compressible/incompressible fluids is modeled as a free boundary which evolves in time. The equations describing the motion of each fluid are supplemented by boundary conditions at the free surface  involving the physical properties of the interface. For instance, in the free-boundary formulation, it is assumed that the interface has an internal surface tension. However, when the interfacial thickness is comparable to the length scale of the phenomena being examined, the free-boundary description breaks down. Diffuse-interface models provide an alternative description where the surface tension is expressed in its simplest form as $\mm{div}K$, i.e., the capillary tension which was introduced by Korteweg in 1901 \cite{korteweg1901forme}. Later, its modern form was derived by Dunn and Serrin \cite{DS1985}.

In the physical view, it can serve as a phase transition model to describe the motion of compressible fluid with capillarity effect. Owing to its importance in mathematics and physics, there has been profuse works for the mathematical theory of the corresponding compressible NSK system, for example, see \cite{MR4412067,DDL2003,MR2805863} for the global(-in-time) weak solutions with large initial data, \cite{MR1810272,MR4201123,MR4099657,MR4354131,MR4188843} for global strong solutions with small initial data, \cite{MR3160440,HL1994,MR2436788,MR2749439} for local(-in-time) strong/classical solutions with large initial data, \cite{MR3168914} for stationary solutions, \cite{F2016} for highly rotating limit, \cite{MR4228314,MR4369180} for the stability of viscous shock wave, \cite{MR4047969} for the maximal $L^p$--$L^q$ regularity theory, \cite{MR4354130,MR4340485,MR4312286} for the decay-in-time of global solutions, \cite{BYZ2014,MR3032985,MR3433629,MR3210149,HYZ2017} for the vanishing capillarity limit, \cite{K2012,K2014,HPZ2018NA} for the nonisentropic case,  \cite{LY2016,SL2019,JX2022}  for the low Mach number limit, \cite{MR3613503,MR2354691,MR3961293,MR2198187} for the inviscid case, and so on.

The incompressible NSK system \eqref{1.1} has been also widely investigated. The local existence of a unique strong solution was obtained by Tan--Wang \cite{TW2010}. Burtea--Charve also established the existence result of strong solutions with small initial perturbations in Lagrangian coordinates \cite{BC2017}, and further presented that the lifespan goes to infinity as the
capillarity number goes to zero. Yang--Yao--Zhu proved that as both capillarity number and viscosity coefficient vanish, local solutions of the Cauchy problem for the incompressible NSK system converge to the one of the inhomogeneous incompressible Euler equations  \cite{YYZ2015}. Recently a similar result was also established by Wang--Zhang for the fluid domain being a horizontally periodic slab with Navier-slip boundary conditions \cite{MR4713202}.
  
In addition, it also has been investigated that the capillarity under the large capillarity number stabilizes the motion of the incompressible fluid endowed with internal capillarity. Bresch--Desjardins--Gisclon--Sart derived that the capillarity slows down the growth rate of linear Rayleigh--Taylor (RT) instability,  which occurs when a denser fluid lies on top of a lighter fluid under gravity \cite{BDG2008}. 
Li--Zhang proved that the capillarity can inhibit RT instability for properly large capillarity number by the linearized motion equations \cite{MR4622412} (Interesting readers further refer to \cite{Na2023,MR4622412,MR4434208,MR4645122} for the existence of RT instability solutions in the incompressible NSK system with small capillarity number). Later, motivated by the result of magnetic tension inhibiting the RT instability in the two-dimensional (2D) non-resistive magnetohydrodynamic (MHD) fluid in \cite{MR4587686}, Jiang--Li--Zhang mathematically proved that the capillarity inhibits RT instability in the 2D NSK equations in the Lagrangian coordinates \cite{JLZa2023}. Recently Jiang--Zhang--Zhang further verified such inhibition phenomenon for the 3D case in Eulerian coordinates \cite{JZZa2024}. Jiang--Zhang--Zhang's result rigorously presents that the capillarity has the stabilizing effect as well as the magnetic tension in the MHD fluid.

It is well-known that the larger the magnetic tension, the stronger the field intensity in the MHD fluid.
Moreover the \emph{local} solutions of the system of the idea MHD fluid with well-prepared initial data tend to a solution of a 2D Euler flow coupled with a linear transport equation as the field intensity goes to infinity, see \cite{JJX2019,Goto1990,Browning1982} for relevant results. This naturally motivates us to except a similar result in the incompressible NSK system. In this paper, under  the well-prepared initial data, we also construct a family of the \emph{global} strong solutions of the 3D incompressible NSK equations with respect to $\kappa$, where the solutions converge to a unique solution of a 2D incompressible Navier--Stokes-like equations as $\kappa$ goes to infinity, see Theorem \ref{thm3}. As Lin pointed out in \cite{Lin2012SomeAI}, the stratified fluids with the internal surface tension, the viscoelastic fluids, the non-resistive MHD fluids and the diffuse-interface model (or NSK model) can also be regarded as complex fluids with elasticity. Our mathematical result in Theorem \ref{thm3} supports that the capillarity has the role of elasticity again. 

\section{Main results}\label{202412081531}

This section is devoted to introducing our main results in details.
Let us consider a rest stat of the system \eqref{1.1}. We choose the equilibrium density $\bar{\rho}$ in the rest state to satisfy
\begin{align}
\label{202501232128}
\bar{\rho}=ax_3+b\mbox{ with given constants }a,\ b>0.
\end{align}
\emph{Without loss of generality, we consider}
\begin{align}
\label{202501232saf128}a=b=1.
\end{align}
Denoting the perturbations around the rest state $(\bar{\rho},0)$ by
$$\varrho=\rho-\bar{\rho},\ \ u=v-0,$$
and then recalling the relation 
\begin{align*} 	
\Delta \rho\nabla\rho=  \nabla(\rho\Delta  \rho) - \rho\nabla \Delta  \rho,\end{align*}
we obtain the following perturbation system from \eqref{1.1} with \eqref{202501232119}:
\begin{equation}\label{1.2}
 \begin{cases}
 \varrho _t+u\cdot\nabla \varrho+ u_3=0,\\
\rho (u_{t}+ u\cdot  \nabla u)+\nabla  \beta =\mu\Delta u-\kappa^2\Delta{\varrho}( \mathbf{e}^3+ \nabla{\varrho} )
	, \\
\mm{div}  u=0,
\end{cases}
\end{equation}
where $\mathbf{e}^3$ represents the unit vector with the third component being $1$ and we have defined that $\beta:= P -\kappa^2\rho\Delta \rho$. 
The initial condition reads as follows
\begin{align}
\label{1.8saf}
(u,\varrho)|_{t=0}=(u^0,\varrho^0).
\end{align}   
Here and in what follows, we always use the right superscript $0$ to emphasize the initial data.  

 We consider that the fluid domain denoted by $\Omega$ is a slab, i.e. 
\begin{align*}
 	\Omega:=\mathbb{R}^2\times(0,h),
 \end{align*} The 2D domain
$\mathbb{R}^2\times \{0,h\}$, denoted by $\partial\Omega$,  is  the  boundary of  $\Omega$.  
We focus on the following Navier(-slip) boundary conditions on the velocity on $\partial\Omega$:
	\begin{gather*}
	u|_{\partial\Omega}\cdot{\mathbf{n}}=0,  \ ((\mathbb{D}u|_{\partial\Omega}){\mathbf{n}})_{\mm{tan}}=0,
	\end{gather*}
	where  $\mathbf{n}$ denotes the unit outward normal to $\Omega$ and the subscript ``tan" means the tangential component of a vector (for example, $u_{\mm{tan}}=u-(u\cdot \mathbf{n})\mathbf{n}$) \cite{DLX2018,MR4342183}. The Navier slip boundary conditions describe
an interaction between a viscous fluid and a solid wall.
	Since $\Omega$ is a slab domain, the Navier boundary conditions are equivalent to the boundary condition
	\begin{align}
	(u_3,\partial_3 u_1,\partial_3 u_2)|_{\partial\Omega}=0.
	\label{20220202081737}
	\end{align}
We mention that it is not clear to us \emph{whether our main results in Theorems \ref{thm2} and \ref{thm3} can be extended to the non-slip boundary condition of velocity}.
 
However, to investigate the asymptotical behavior of solutions with respect to capillarity number at infinity and to except the vanishing of $\varrho$, we set 
\begin{align*}
\varrho=\kappa  ^{-1}\sigma  
\end{align*}
 and then transform \eqref{1.2} into
\begin{align}\label{1.7}
\begin{cases}
\sigma_t+u\cdot\nabla \sigma+\kappa  u_3=0,\\
\rho (u_t+ u\cdot\nabla u)
+\nabla\beta
=\mu\Delta u-\Delta \sigma(\kappa \mathbf{e}^3+\nabla \sigma),\\
\mm{div}  u=0.
\end{cases}
\end{align}  
The corresponding initial condition reads as follows
\begin{align}
\label{1.8x}
(u,\sigma)|_{t=0}=(u^0,\sigma^0).
\end{align}    In addition, to avoid the vacuum of density, we assume that 
 \begin{align}
d\leqslant \inf\limits_{x\in\Omega}\big\{{\rho}^{0}(x)\big\} ,\mbox{ where } {\rho}^{0}:=\bar{\rho}+\kappa  ^{-1}\sigma^0\mbox{ and }d\in \mathbb{R}^+.
\label{202412051008}
\end{align} 
  Before stating our main results, we shall introduce simplified notations, which will be used throughout this paper.

 (1) Simplified basic notations:   $I_a:=(0,a)$ denotes an  interval, in particular, $I_\infty=\mathbb{R}^+=(0,\infty)$.  $\overline{S}$ denotes the closure of a set $S\subset \mathbb{R}^n$ with $n\geqslant 1$, in particular, $\overline{I_T} =[0,T]$ and $\overline{I_\infty} = \mathbb{R}^+_0$. $a\lesssim b$ means that $a\leqslant cb$ for some constant $c>0$, where $c >0$ at most depends on  the domain $\Omega$, $\bar{\rho}$, $\mu$, $d$ in \eqref{202412051008}, and may vary from line to line. Sometimes we also denote $c$ by $c_i$ for emphasizing that $c_i$ is fixed, where $i=1$, $2$. $\partial_i:=\partial_{x_i}$, where $1\leqslant i\leqslant 3$. Let $f:=(f_1,f_2,f_3)^{\top}$ be a vector function defined in a 3D domain, we define that $f_{\mm{h}}:=(f_1,f_2)^{\top}$ and
$\mm{curl}{f}:=(\partial_{2}f_3-\partial_{3}f_2,
\partial_{3}f_1-\partial_{1}f_3,\partial_{1}f_2-\partial_{2}f_1)^{\top}$.  $\nabla^{\bot}=(-\partial_2,\partial_1,0)$, $\nabla_{\mm{h}}:=(\partial_{1},\partial_{2})^{\top}$, $\nabla^{\perp}_{\mm{h}}:=(-\partial_{2},\partial_{1})^{\top}$, $\mm{div}_{\mm{h}}\cdot:=\partial_1\cdot +\partial_2\cdot$ and  $\Delta_\mm{h}:=\partial^2_{1}+\partial^2_{2}$. 
	  For the simplicity, we denote $\sqrt{\sum_{i=1}^n\|w_i\|_X^2}$ by $\|(w_1,\cdots,w_n)\|_X$, where $\|\cdot\|_X$ represents a norm, and $w_i$ is scalar function or vector function for $1\leqslant i\leqslant n$.	
	
(2) Simplified Banach spaces, norms and semi-norms:
\begin{align}
&L^p:=L^p (\Omega)=W^{0,p}(\Omega),\
{H}^i:=W^{i,2}(\Omega ), \
{H}^i_{\mm{loc}}:=W^{i,2}_{\mm{loc}}(\Omega ), \
{H}^j_0:=\{\phi\in H^j~|~\phi|_{\partial\Omega}=0\},  \nonumber \\[1mm]
&{ {H}}^3_{\bar{\rho}}:=\{\phi\in H^3_0~|~\partial^2_{3}\phi|_{\partial\Omega}=0, \partial_{3}\phi |_{\partial\Omega}=-\kappa\bar{\rho}'|_{\partial\Omega}\}, \ \mathcal{H}^i:=\{u\in H^i~|~\mm{div}u=0,\ u_3|_{\partial\Omega}=0\},\nonumber\\ 
&    \mathcal{H}^k_{\mathrm{s}}:=\{w\in \mathcal{H}^k~|~ \partial_3w_1 =\partial_3w_2=0\mbox{ on }\partial\Omega \},\  \|\cdot \|_i :=\|\cdot \|_{H^i},\    \|\cdot\|_{{i},l}:=\sum_{\alpha_1+\alpha_2=i} \|\partial_{1}^{\alpha_1}\partial_{2}^{\alpha_2}\cdot\|_l, \nonumber
\end{align}
where $1\leqslant p\leqslant \infty$ is a real number, and $i$, $l \geqslant 0$, $j\geqslant 1$, $k\geqslant 2$ are integers.

(3) Simplified spaces of functions with values in a Banach space:
\begin{align} 
&   {\mathfrak{P}} _{T}:=\{\sigma\in
C^0(\overline{I_T},   { {H}}^3_{\bar{\rho}})~|~ \sigma_t\in C^0(\overline{I_T},H^1)\cap  L^2(I_T,{H}^2)\}, \nonumber\\
& {\mathcal{V}}_{ T}: =  \{u\in C^0(\overline{I_T},{\mathcal{H}^2_{\mathrm{s}}} )\cap L^2(I_T,{H}^3) ~|~
u_t\in C^0(\overline{I_T},L^2 )\cap  L^2(I_T,{H}^1)\}. \nonumber
\end{align}
			
(4) Energy/dissipation functionals (generalized):
\begin{align*}
&E(t):=\|(   \sigma \|_3^2
+\|u\|_2^2
+\| \kappa   u_3\|_{1}^2
+\|\kappa   (\Delta_{\mm{h}} \sigma,\nabla_{\mm{h}} \partial_3\sigma)\|_{0}^2,\ E^0:=E(0) ,\\
&\mathcal{E}(t):=   \|\sigma_{t}\|_1^2 +\|u_t\|_0^2+ E(t),\ \mathcal{D}(t):=   \| \sigma\|^2_{1,2}+\|(\sigma_t,\nabla u)\|^2_{2}+\|u_t\|^2_1. 			\end{align*}
	 
Now we state the first result, which presents that the initial-boundary value problem \eqref{20220202081737}--\eqref{1.8x} with small initial data admits a unique global strong solution with uniform estimates with respect to $\kappa  $. 
 \begin{thm}\label{thm2}
Let  $\mu$, $\kappa$, $d >0$ be given, and $\bar{\rho}$ be defined by \eqref{202501232128}. There exist constants $c_1$, $c_2$ and $\chi$, where
\begin{align}c_1,\ c_2\geqslant 1\mbox{ and } \chi:= \max\{\kappa  ^{-1},c_2\}  ,
\label{2022412082055}
\end{align} 
 such that, for any $(\sigma^0,u^0) $ satisfying \eqref{202412051008},
 \begin{align}
&(\sigma^0,u^0)\in H^3_{\bar{\rho}}\times {\mathcal{H}^2_{\mathrm{s}}}\mbox{ and } E^0\leqslant  (3c_1 \chi^9)^{-3},
\label{2022412051007}
\end{align}  
  the initial-boundary value problem \eqref{20220202081737}--\eqref{1.8x} has a unique global strong solution 
  $$(\sigma,u,\nabla\beta)\in  {\mathfrak{P}} _{\infty}\times {\mathcal{V}}_{\infty}\times C^0(\mathbb{R}_0^+,L^2).$$
   Moreover the solution satisfies the following  estimates
\begin{align}\label{1.12}
&\mathcal{E}(t)+\int_0^t\mathcal{D}(\tau)d\tau\leqslant 2c_1\chi^9E^0< 1 ,\\
 &    \|\nabla\mathcal{P}(t)\|_{0}^2\lesssim  c_1\chi^9(1+\chi^2  ){E^0}\mbox{ for any }t>0 
 \label{202412061425}
            \end{align}
and 
\begin{align}
       & \|(\kappa   (\Delta_{\mm{h}} \sigma,\nabla_{\mm{h}}\partial_{3}\sigma), \nabla\mathcal{P} )\|_{L^2( I_T,H^1)}^2 \lesssim  c_1\chi^9(1+\chi^2 +T) {E^0}\mbox{ for any }T>0,\label{1sdaf2}
\end{align}
where $\mathcal{P}:=\beta+\kappa   \partial_3\sigma$.
\end{thm}

The existence of strong solutions with small initial perturbations to the initial-boundary value problem \eqref{1.2}--\eqref{20220202081737} in a horizontally periodic domain $\Omega_{\mm{p}}:=\mathbb{T}^2\times (0,h)$ has been established by Jiang--Zhang--Zhang, where $\mathbb{T}:=\mathbb{R}/\mathbb{Z}$. Let us first roughly recall some key ideas in Jiang--Zhang--Zhang's proof in \cite{JZZa2024}.

The standing point of  Jiang--Zhang--Zhang's proof is the basic energy identity in differential version for 
the problem \eqref{1.2}--\eqref{20220202081737}:
\begin{align}\label{2.sadf16}
\frac{1}{2}\frac{\mm{d}}{\mm{d}t}(\|\kappa \nabla \varrho\|_{L^2(\Omega_{\mm{p}})}^2+\|\sqrt{\rho}u\|_{L^2(\Omega_{\mm{p}})}^2)+\mu\|\nabla u\|_{L^2(\Omega_{\mm{p}})}^2=0.
\end{align}  
We call \eqref{2.sadf16} the  zero-order energy estimate of $(\varrho,u)$. 
Then they further derived  energy estimates for the high-order spacial derivatives and the temporal derivatives of $(\varrho,u)$. However, the integrals related to the nonlinear terms in the system \eqref{1.2} appear in the high-order energy estimates. We call such integrals the nonlinear integrals for simplicity. In particular, there exist a troublesome nonlinear integral $\int(\partial_3^3\varrho)^2\partial_3u_3\mm{d}x$, which is equal to $-\int(\partial_3^3\varrho)^2 \mm{div}_{\mm{h}}u_{\mm{h}} \mm{d}x$ due to the incompressible condition \eqref{1.2}$_3$.  Jiang--Zhang--Zhang naturally expect that  $\nabla_{\mm{h}}u$ enjoy a fine decay-in-time, which contributes to close the troublesome nonlinear integral.
However one can not directly derive the decay-in-time of $\nabla_{\mm{h}}u$ from high-order energy estimates due to the absence of the dissipation of $\nabla_{\mm{h}}\varrho$. Fortunately, they can capture the dissipation of  $\nabla_{\mm{h}}\varrho$ from the vortex equations, which can be obtained by applying $\mm{curl}$ to \eqref{1.2}$_2$.  The dissipation of  $(\nabla_{\mm{h}}\varrho,\nabla_{\mm{h}}u)$, together with the horizontal periodicity of $\Omega_{\mm{p}}$,  results in the decay-in-time of $\nabla_{\mm{h}}\varrho$ and $\nabla_{\mm{h}}u$ by energy method with extremely fine estimates. 

Following Jiang--Zhang--Zhang's  ideas in \cite{JZZa2024} for the initial-boundary vaule problem \eqref{20220202081737}--\eqref{1.8x}, however it seems to be difficult to establish the desired decay-in-time of $\nabla_{\mm{h}}u$   due to the unboundedness of the fluid domain $\Omega$. This results in that we shall develop a new alternative method to estimate for the troublesome  nonlinear integral $\int(\partial_3^3\sigma)^2\partial_3u_3\mm{d}x$  (recalling the linear relation $\sigma=\kappa\varrho$). More precisely,  we use the transport equation \eqref{1.7}$_1$ twice and the anisotropic Gagliardo--Nirenberg--Sobolev type estimate in \ref{lemma4.7.2}  to estimate $\int(\partial_3^3\sigma)^2\partial_3u_3\mm{d}x$,  see \eqref{2.29} for the detailed derivation. Based on this new idea, we can refine the energy method in \cite{JZZa2024} to establish Theorem \ref{thm2}.  

It is easy see from \eqref{2022412082055} that  $\chi$ for $\kappa  \geqslant 1$ reduces to 
$\chi= c_2$. Consequently, we can make use of the uniform-in-$\kappa  $  estimates in \eqref{1.12}--\eqref{1sdaf2} with $\kappa  \geqslant 1$
to establish the asymptotical behavior with respect to capillarity number at infinity. More precisely, for any given $\kappa  >0$, we choose initial data $\sigma^0_\kappa  $ and $u^0_\kappa  $, which satisfy 
 \eqref{202412051008} and  \eqref{2022412051007} with $(\sigma^0_\kappa  ,u^0_\kappa  ) $ in place of $(\sigma^0 ,u^0 )$. In view of Theorem \ref{thm2}, the  initial-boundary value problem \eqref{1.2}--\eqref{20220202081737} with $(\varrho^0_\kappa  ,u^0_\kappa  ) $ in place of $(\varrho^0 ,u^0 )$ admits a global strong solution denoted by $(\varrho^\kappa   ,u^\kappa  )$, where we have defined that $\varrho^0_\kappa :=\kappa^{-1}\sigma^0_\kappa $. Moreover  the solutions have the following asymptotical behavior with respect to
$\kappa  $ at infinity
\begin{thm}\label{thm3}
Let $p \geqslant 1$. We additionally assume that $(u^0_\kappa)_{\mm{h}}\to w^0_{\mm{h}}$ in $H^2$ as $\kappa\to \infty$. There exist functions ${\tilde{u}}$, $\mathcal{Q}$, $\mathcal{M} $ and 
$\mathcal{N}_3$ such that, for any given $T>0$,
\begin{align} 
&u^\kappa  \rightarrow{\tilde{u}}\mbox{ weakly* in }L^{\infty}(I_T, {\mathcal{H}^2_{\mathrm{s}}})\mbox{ and weakly in }L^2(I_T, H^3)\mbox{ wit }\tilde{u}_3=0,\label{1.16}\\
&u^\kappa  \rightarrow{\tilde{u}}\mbox{ strongly in }L^{p}(I_T,H^2_{\mm{loc}})\cap C^0(\overline{I_T},H^1_{\mm{loc}}),\label{3.sdf9} \\
&u_t^\kappa  \rightarrow \tilde{u}_t\mbox{ weakly* in } L^{\infty}(I_T,L^{2})\mbox{ and weakly in } L^2(I_T,{\mathcal{H}^1}), \\
&(\varrho^\kappa , \varrho^\kappa _t)\to (0,0), \mbox{ strongly in }L^{\infty}(I_T,H^3)\times  L^2(I_T,H^2),\\
&\nabla\mathcal{P}^\kappa  \rightarrow \nabla\mathcal{Q}\mbox{ weakly* in } L^{\infty}(I_T,L^{2})\mbox{ and weakly in } L^2(I_T,H^1),\\
&\kappa  ((\nabla_{\mm{h}}\partial_3\sigma^\kappa)^{\top}  ,-\Delta_{\rm h} \sigma^\kappa  )^\top\rightarrow (\nabla_{\mm{h}}\mathcal{M}^{\top} ,\mathcal{N}_3)^\top\nonumber \\
&\mbox{weakly* in } L^{\infty}(I_T,L^2)\mbox{ and weakly in }L^2( I_T, H^1_0),\label{3asd.6}
\end{align} as $\kappa  \rightarrow\infty$. 
Moreover,  ${\tilde{u}}_{\mm{h}}$, $\mathcal{Q}$, $\mathcal{M} $  and 
$\mathcal{N}_3$ satisfy
\begin{align}\label{1.13}
&\begin{cases}
\bar{\rho}(\partial_t{\tilde{u}}_{\mm{h}}+{\tilde{u}}_{\mm{h}}\cdot\nabla_{\mm{h}} {\tilde{u}}_{\mm{h}})+\nabla_{\rm h}(\mathcal{Q}-\mathcal{M})-\mu\Delta_{\mm{h}}{\tilde{u}}_{\mm{h}}
=\mu\partial_3^2{\tilde{u}}_{\mm{h}},\\
\mm{div}_{\mm{h}}{\tilde{u}}_{\mm{h}}=0,\\
{\tilde{u}}_{\mm{h}}|_{t=0}=w_{\mm{h}}^0 \mbox{ with }\mm{div}_{\mm{h}}w_{\mm{h}}^0=0 
\end{cases}
\end{align}
and 
\begin{align}\label{1z13}
\begin{cases}
- \Delta \mathcal{Q}=
\bar{\rho}\nabla {\tilde{u}}_{\mm{h}}:\nabla_{\mm{h}}\tilde{u}_{\mm{h}},\ \partial_3\mathcal{Q}|_{\partial\Omega}=0,\\
-\Delta_{\mm{h}} \mathcal{M}=\partial_3^2\mathcal{Q},\\
\mathcal{N}_3 =\partial_3\mathcal{Q} .
\end{cases}
\end{align}
\end{thm}
\begin{rem}
Modifying the proofs of Theorems \ref{thm2} and \ref{thm3} by further using the energy method with time-weight in \cite{JZZa2024}, we also obtain the correspond results for the case of the horizontally periodic domain with finite height, i.e. $2\pi L_1\mathbb{T}\times 2\pi L_2\mathbb{T}\times (0,h) $. In addition,
since $\chi\to \infty$ for $\kappa  \to 0$, we can not except the vanishing capillarity limit from Theorem \ref{thm2}. 
\end{rem} 
 
We will prove  the  asymptotical behavior in Theorem \ref{thm3} by exploiting a compactness argument, the details of which will be presented in Section \ref{subsec:04}.

\section{Proof of Theorem \ref{thm2}}\label{subsec:02}
This section is devoted to the proof of the  global(-in-time) solvability with uniform-in-$\kappa  $ estimates for the initial-boundary value problem \eqref{20220202081737}--\eqref{1.8x}. The key point is to \emph{a priori} derive the uniform-in-$\kappa  $ estimate in \eqref{1.12}. For this purpose, let $T>0$ be a fixed time and $(\sigma,u)$ a solution to \eqref{20220202081737}--\eqref{1.8x} on $\Omega\times I_T$ with initial data $(\sigma^0,u^0)\in { {H}}^3_{\bar{\rho}}\times {\mathcal{H}^2_{\mathrm{s}}}$. Moreover, we assume that $\sigma^0$ satisfies \eqref{202412051008} and  the solution $(\sigma,u)$ is sufficiently regular so that the procedure of formal deduction makes sense.

\subsection{Basic estimates}\label{Ene}

This section is devoted to deriving some basic estimates  of $(\sigma,u)$ from the initial-boundary value problem  \eqref{20220202081737}--\eqref{1.8x}. Let us first recall the energy identity
of $(\sigma,u)$.
\begin{lem}\label{lem2.2}
It holds that
\begin{align}\label{2.16}
\frac{1}{2}\frac{\mm{d}}{\mm{d}t}(\| \nabla \sigma\|_0^2+\|\sqrt{\rho}u\|_0^2)+\mu\|\nabla u\|_0^2=0,
\end{align} 
where $\rho=\bar{\rho}+\sigma/\kappa$ and $\bar{\rho}$ is defined by \eqref{202501232128} with $a=b=1$ in \eqref{202501232saf128}.
\end{lem}
\begin{proof}
 Taking the inner product of \eqref{1.7}$_1$ and $-\Delta \sigma$ in $L^2$, and then exploiting the integration by parts and the  boundary condition of $\sigma$ in \eqref{2.1}, we have
\begin{align*}
\frac{1}{2}\frac{\mm{d}}{\mm{d}t}\int| \nabla \sigma|^2\mm{d}x=\int\Delta \sigma(\kappa   u_3+   u\cdot\nabla \sigma)\mm{d}x.
\end{align*}

Similarly, taking the inner product of \eqref{1.7}$_2$ and $u$ in $L^2$, and then making using the integration by parts, the mass equation $\eqref{1.1}_1$, the boundary condition \eqref{20220202081737} and the incompressible condition \eqref{1.7}$_3$, we obtain
\begin{align*}
&\frac{1}{2}\frac{\mm{d}}{\mm{d}t}\int\rho|u|^2\mm{d}x+\mu\int|\nabla u|^2\mm{d}x\nonumber\\
=&\frac{1}{2}\int\rho_t|u|^2\mm{d}x-\int(\Delta{\sigma}(\kappa\mathbf{e}^3+ \nabla{\sigma}) + {\rho}u\cdot\nabla u)\cdot u \mm{d}x\nonumber\\
=&-\int\Delta{\sigma}(\kappa u_3+   \nabla{\sigma}\cdot u)\mm{d}x.
\end{align*}
Summing up the above two identities yields \eqref{2.16}.
\end{proof}

Next we further extend \eqref{2.16} to both the cases satisfied by the highest-order spacial derivatives and the temporal derivatives of $(\sigma,u)$ resp.. 
\begin{lem}\label{lem2.5}
	It holds that
	\begin{align}\label{2.17}
	&\frac{1}{2}\frac{\mm{d}}{\mm{d}t}\left(\| \nabla\Delta \sigma\|_0^2+\|\sqrt{\rho}\Delta u\|^2_{0}
-3 \int\left(\kappa ^{-1}-\frac{7}{2\kappa ^2}\partial_3\sigma \right)\partial_3\sigma(\partial_3^3\sigma)^2\mm{d}x
\right)
	+\mu \|\nabla\Delta u\|^2_{0}\nonumber\\
&    \lesssim   (1+\kappa  ^{-1}\| \sigma\|_3)\|\nabla u\|_1
(\|u_t\|_1+ \|u\|_2\|\nabla u\|_2)
+\| \sigma\|_{1,2}(\| \sigma\|_3+\kappa  ^{-2}\| \sigma\|_3^3)\|\nabla u\|_2 
	\end{align}	
and
	\begin{align}\label{2.42}
	&\frac{\mm{d}}{\mm{d}t}(\| \nabla \sigma_t\|_0^2+\|\sqrt{{\rho}}u_t\|^2_{0})+\mu\|\nabla u_t\|^2_{0}\nonumber\\
	&\lesssim\| \sigma_t\|_2\left(\| \sigma_t\|_1\|u\|_2+(\| \sigma\|_3+ \kappa  ^{-1} \|u\|_2^2 )
\|u_t\|_0\right)\nonumber\\
&\quad +(\kappa^{-1}\|\sigma_t\|_1+(1+\kappa  ^{-1}\| \sigma\|_2)\|u\|_2)\|u_t\|_0 \|u_t\|_1 .
\end{align}  
\end{lem}
 \begin{proof}
  (1) Let $1\leqslant i\leqslant 3$. Applying $\Delta$ and $\partial_i$ to \eqref{1.7}$_1$ and \eqref{1.7}$_2$ resp., we obtain
	\begin{align}\label{2.18}
	\begin{cases}
	\Delta(\sigma_t+u\cdot\nabla \sigma+ {\kappa  } u_3)=0,\\
		\rho (\partial_i u_{t}+ u\cdot\nabla \partial_i  u)+\nabla \partial_i\beta \\
 = \partial_i(\mu\Delta u-\Delta   {\sigma}(\kappa\mathbf{e}^3+  \nabla \sigma)) 	- \partial_i \rho u_t - \partial_i(\rho u)\cdot\nabla u. 
	\end{cases}
	\end{align} 
	Taking the inner product of \eqref{2.18}$_1$ and $-\Delta^2 \sigma$ in $L^2$, and then  using the integration by parts and the boundary conditions of $(\sigma,\partial_3^2\sigma,u_3)$ in \eqref{20220202081737}, \eqref{2.1} yields
	\begin{align}\label{2.19}
	\frac{1}{2}\frac{\mm{d}}{\mm{d}t}\| \nabla\Delta \sigma\|^2_{0}
	=\int(\Delta^2 \sigma\Delta ( u\cdot\nabla \sigma)-\kappa\nabla \Delta  \sigma \cdot \nabla \Delta u_3 
	 )\mm{d}x.
	\end{align}

Exploiting the integration by parts, \eqref{1.1}$_1$ and the boundary condition of $u_3$ in \eqref{20220202081737}, we can derive that
	\begin{align}
\label{202501252040}
	\int\rho u\cdot \nabla \Delta  u\cdot \Delta u\mm{d}x
	=-\frac{1}{2}\int\mm{div} (\rho u)| \Delta u|^2\mm{d}x
	=\frac{1}{2}\int\rho_t| \Delta u|^2\mm{d}x.
	\end{align}
In view of the incompressible condition, the boundary condition of $(\varrho, \partial_3^2\varrho,u_3,\partial_3u_{\mm{h}})$ in \eqref{20220202081737}, \eqref{2.1} and the equation \eqref{1.7}$_2$, it is to see that 
\begin{align}\label{20225012412348}
(\partial_3^2 u_3, \partial_3  \beta)|_{\partial\Omega}=0.
\end{align}
Taking the inner product of \eqref{2.18}$_2$ and $-\partial_i \Delta u$ in $L^2$, and then making use of   the incompressible condition, the integration by parts, the boundary conditions of $(\sigma, \partial_3\rho, u_3,\partial_3u_{\mm{h}},\partial_3^2 u_3, \partial_3\beta)$ in \eqref{20220202081737},  \eqref{20225012412348},  \eqref{2.1}, and the identity \eqref{202501252040}, we deduce that 
	\begin{align}
&\frac{1}{2}\frac{\mm{d}}{\mm{d}t}\|\sqrt{\rho} \Delta u\|^2_{0}+\mu\|\nabla\Delta  u\|^2_{0}\nonumber\\
	=&\int\bigg( \partial_i ( \Delta \sigma(\kappa   \mathbf{e}^3 +\nabla \sigma))\cdot \Delta \partial_i u-(2\nabla  \rho\cdot \nabla u_t \nonumber 
\\ \label{2.20}	
&+\Delta \rho u_t+2 \partial_i(\rho u)\cdot\partial_i\nabla u+\Delta (\rho u)\cdot\nabla u)\cdot \Delta u\bigg)\mm{d}x.
	\end{align} 
Putting \eqref{2.19}  and   \eqref{2.20} together yields
	\begin{align}\label{2.21}
	&\frac{1}{2}\frac{\mm{d}}{\mm{d}t}(\| \nabla\Delta \sigma\|_0^2+\|\sqrt{\rho}\Delta u\|^2_{0})
	+\mu\|\nabla\Delta u\|^2_{0}\nonumber\\
    =& \int( \partial_i\Delta u\cdot\partial_i(\Delta \sigma\nabla \sigma)+\Delta^2 \sigma\Delta(u\cdot\nabla \sigma) )\mm{d}x -\int(2\nabla\rho\cdot\nabla u_t\nonumber\\
&+\Delta\rho u_t+2\partial_i(\rho u)\cdot\nabla\partial_i u+\Delta(\rho u)\cdot\nabla u)\cdot \Delta u\mm{d}x
=:J_2-J_1,
	\end{align}	
 where we have used the Einstein convention of summation over repeated indices.	
	Next we shall estimate for $J_1$ and $J_2$ in sequence.

Using the definition of $\rho$ and the imbedding inequality \eqref{1}, we get
\begin{align}
\|\nabla^i\rho\|_{L^{\infty}}=\|\nabla^i(\bar{\rho}+\kappa  ^{- 1}\sigma)\|_{L^{\infty}}\lesssim
1+\kappa  ^{-1 } \| \nabla^i \sigma\|_{2}\mbox{ for }i=0,\ 1\label{2.44}.
\end{align} 
Making use of H\"older inequality,    the product estimates in \eqref{4} and  \eqref{2.44}, we can estimate that
	\begin{align}\label{2.26}
	J_1\lesssim(1+\kappa  ^{-1}\| \sigma\|_3)\|\nabla u\|_1(\|u_t\|_1+\|u\|_2\|\nabla u\|_2).
	\end{align}

Now we turn to estimating for $J_2$.
Using the incompressible condition, the integration by parts and the boundary conditions of $(\sigma,\partial_3^2\sigma ,u_3,\partial_3u_{\mm{h}},\partial_3^2u_3)$ in \eqref{20220202081737}, \eqref{20225012412348}, \eqref{2.1}, it holds that
\begin{align*}
&\int \Delta^2\sigma u\cdot\nabla\Delta \sigma\mm{d}x
=-\int(\nabla \Delta\sigma\cdot \nabla )u\cdot\nabla\Delta\sigma\mm{d}x,\\
& \int\Delta^2\sigma \partial_j u\cdot\nabla\partial_j\sigma\mm{d}x
=- \int\partial_i\Delta\sigma(\partial_j u\cdot \nabla\partial_i\partial_j\sigma +\partial_i\partial_j u\cdot\nabla\partial_j\sigma)\mm{d}x.
\end{align*}
and 
\begin{align*} 
\int \Delta^2 \sigma \Delta u\cdot\nabla \sigma\mm{d}x =-\int \nabla \Delta \sigma \cdot \nabla (\Delta u\cdot\nabla \sigma)\mm{d}x.
\end{align*}
We can use the above three identities and the integration by parts to rewrite $J_2$ as follows:
	\begin{align}
	J_2=& \int( \Delta \sigma\nabla \partial_i\sigma \cdot \partial_i\Delta u -\partial_i \Delta\sigma(\partial_i u\cdot\nabla\Delta\sigma +2(\partial_j u\cdot \nabla\partial_i\partial_j\sigma \nonumber  \\
&+\partial_i\partial_j u\cdot\nabla\partial_j\sigma) +   \Delta u\cdot\nabla \partial_i\sigma)
 )\mm{d}x\nonumber  \\
	=& -  \int(
\partial_i\Delta\sigma(\partial_i u\cdot\nabla\Delta\sigma+2( \partial_j u\cdot\nabla\partial_i\partial_j\sigma+ \Delta u\cdot\nabla\partial_i\sigma \nonumber \\
&+\partial_i\partial_j u\cdot\nabla\partial_j\sigma))+\Delta\sigma\Delta u\cdot\nabla\Delta\sigma )\mm{d}x 
	= J_{2,1}+J_{2,2}+J_{2,3}, \label{2337}
	\end{align}
where we have defined that
\begin{align*}
J_{2,1}:=&- \int\bigg( \sum_{i=1}^2
\partial_i\Delta\sigma(\partial_i u\cdot\nabla\Delta\sigma+  2(\partial_j u\cdot\nabla\partial_i\partial_j\sigma + \Delta u\cdot\nabla\partial_i\sigma  \nonumber \\
&+\partial_i\partial_j u\cdot\nabla\partial_j\sigma))
+ 
\partial_3\Delta_{\mm{h}}\sigma(\partial_3 u\cdot\nabla\Delta\sigma+ 2(    \partial_j u\cdot\nabla\partial_3\partial_j\sigma+\partial_3\partial_j u\cdot\nabla\partial_j\sigma \nonumber \\
&+   \Delta u\cdot\nabla\partial_3\sigma)) + 
\partial_3^3\sigma(\partial_3 u_{\mm{h}}\cdot\nabla_{\mm{h}}\Delta\sigma+\partial_3 u_3\partial_3\Delta_{\mm{h}}\sigma+ 2(\partial_j u_{\mm{h}}\cdot\nabla_{\mm{h}}\partial_3\partial_j\sigma\nonumber \\
&+   \nabla_{\mm{h}} u_3 \cdot\nabla_{\mm{h}}\partial_3^2 \sigma+\partial_3\partial_j u_{\mm{h}}\cdot\nabla_{\mm{h}}\partial_j\sigma+   \partial_3\nabla_{\mm{h}} u_3\cdot\partial_3\nabla_{\mm{h}}\sigma+    \Delta u_{\mm{h}}\cdot\nabla_{\mm{h}}\partial_3\sigma))\nonumber \\
&+\Delta_{\mm{h}}\sigma\Delta u\cdot\nabla\Delta\sigma+\partial_3^2\sigma(\Delta u_{\mm{h}}\cdot\nabla_{\mm{h}}\Delta\sigma+  \Delta u_3\Delta_{\mm{h}}\partial_3\sigma)\bigg)\mm{d}x,\\
J_{2,2}:=&- \int\partial_3^2 \sigma  \partial_3^3\sigma(2\partial_3^2u_3+ 3\Delta u_3) \mm{d}x\mbox{ and }
J_{2,3}:= -3 \int(\partial_3^3 \sigma)^2 \partial_3u_3\mm{d}x.
\end{align*}
	
Take a similar procedure as \eqref{2.26}, we can easily obtain
	\begin{align}\label{2.27}
		J_{2,1}\lesssim\| \sigma\|_{1,2}\| \sigma\|_3\|\nabla u\|_2.
	\end{align}
	
Making use of  the incompressible condition, the integration by parts, the product estimate, and the boundary condition  of $ \partial_3^2\sigma $ in   \eqref{2.1}, we deduce that
	\begin{align}\label{2.28}
	J_{2,2}=&\frac{1}{2} \int(\partial_3^2\sigma)^2 \partial_3(2\partial_3^2u_3+3\Delta u_3  )
\mm{d}x 
=- \frac{1}{2} \int( \partial_3^2\sigma)^2( 2\partial_3^2 \mm{div}_{\mm{h}} u_{\mm{h}}+3\Delta\mm{div}_{\mm{h}} u_{\mm{h}})\mm{d}x\nonumber\\
=&\frac{1}{2}   \int ( 2\partial_3^2u_{\mm{h}} +3\Delta u_{\mm{h}}) \cdot \nabla_{\mm{h}}(\partial_3^2\sigma)^2\mm{d}x 
	\lesssim \| \sigma\|_{1,2}\| \sigma\|_3\|\nabla u\|_{2}.
	\end{align}
	
Due to the absence of the dissipation of $\partial_3^3\sigma$, we can not directly estimate for $J_{2,3}$. To overcome this difficulty, we shall use equation $\eqref{1.7}_1$ twice as follows. 
	\begin{align}
	&J_{2,3}/3 \nonumber \\
=&{\kappa  }^{-1}\int(\partial_3^3 \sigma)^2\partial_3(\sigma_t+u\cdot\nabla \sigma)\mm{d}x\nonumber \\
=& {\kappa  }^{-1}\int((\partial_3^3 \sigma)^2\partial_3(u\cdot\nabla \sigma)	-{2} \partial_3\sigma\partial_3^3\sigma \partial_3^3\sigma_t)\mm{d}x+{\kappa  }^{-1}\frac{\mm{d}}{\mm{d}t}\int\partial_3\sigma(\partial_3^3 \sigma)^2\mm{d}x\nonumber \\
=&{\kappa  }^{-1}\int({2} \partial_3\sigma\partial_3^3\sigma \partial_3^3({\kappa  } u_3 +u\cdot\nabla \sigma  ) +(\partial_3^3 \sigma)^2 \partial_3(u\cdot\nabla \sigma))\mm{d}x+{\kappa  }^{-1}\frac{\mm{d}}{\mm{d}t}\int\partial_3\sigma(\partial_3^3 \sigma)^2\mm{d}x\nonumber \\
=&   {\kappa  }^{-1}\int ({2}\partial_3\sigma\partial_3^3\sigma (
 {\kappa  }\partial_3^3u_3 +\partial_3^3u\cdot\nabla \sigma
	+3\partial_3^2u\cdot\nabla\partial_3 \sigma\nonumber \\
&+3\partial_3u\cdot\nabla\partial_3^2 \sigma ) +(\partial_3^3 \sigma)^2 \partial_3u\cdot\nabla \sigma )\mm{d}x+{\kappa  }^{-1}\frac{\mm{d}}{\mm{d}t}\int\partial_3\sigma(\partial_3^3 \sigma)^2\mm{d}x
\nonumber \\
=&{\kappa  }^{-1}\int\partial_3^3\sigma(2\partial_3\sigma(\partial_3^3u_{\rm h}\cdot\nabla_{\rm h} \sigma
	+3\partial_3^2u_{\rm h}\cdot\nabla_{\rm h}\partial_3 \sigma+3\partial_3u_{\rm h}\cdot\nabla_{\rm h}\partial_3^2 \sigma\nonumber \\
	&+ \partial_3^3\sigma\partial_3u_{\rm h}\cdot\nabla_{\rm h} \sigma))\mm{d}x+2\int
\partial_3\sigma\partial_3^3\sigma (\partial_3^3u_3+{\kappa  }^{-1} (\partial_3 \sigma\partial_3^3u_3
	+3\partial_3^2 \sigma\partial_3^2u_3 ))\mm{d}x\nonumber \\
	&+{7}{\kappa  }^{-1 }\int\partial_3\sigma(\partial_3^3 \sigma)^2\partial_3u_3\mm{d}x+{\kappa  }^{-1}\frac{\mm{d}}{\mm{d}t}\int\partial_3\sigma(\partial_3^3 \sigma)^2\mm{d}x
\nonumber 	\\
=:&J^{2,3}_1+2J^{2,3}_2+7J^{2,3}_3+{\kappa  }^{-1}\frac{\mm{d}}{\mm{d}t}\int\partial_3\sigma(\partial_3^3 \sigma)^2\mm{d}x \label{202412171632}
,
	\end{align}
	where in the fourth equality we have used the identity
	\begin{align*}
		 \int\partial_3^3 \sigma( \partial_3^3 \sigma u\cdot\nabla\partial_3 \sigma+2\partial_3 \sigma u\cdot\nabla\partial_3^3 \sigma)\mm{d}x= \int u\cdot\nabla( \partial_3\sigma(\partial_3^3 \sigma)^2)\mm{d}x=0.
	\end{align*}
Next we shall estimate the three terms $J^{2,3}_1$--$J^{2,3}_3$ in sequence.

It  follows from the product estimates that
\begin{align}
J^{2,3}_1\lesssim\kappa  ^{-1}\| \sigma\|_{1,2}\| \sigma\|_3^2\|\nabla u\|_{2}.
\label{20250124}
\end{align}

Exploiting the incompressible condition, the integration by parts, the product estimates and the boundary condition  of $ \partial_3^2\sigma $ in  \eqref{2.1}, we have
\begin{align}
J^{2,3}_2=&-\int
\partial_3^3\sigma (\partial_3\sigma\partial_3^2\mm{div}_{\mm{h}}u_{\mm{h}}+{\kappa  }^{-1} \partial_3 \sigma(\partial_3 \sigma\partial_3^2\mm{div}_{\mm{h}}u_{\mm{h}}
	+3\partial_3^2 \sigma\partial_3 \mm{div}_{\mm{h}}u_{\mm{h}} ))\mm{d}x\nonumber \\
    	=& \int
\partial_3^3\sigma (\partial_3^2u_{\mm{h}}\cdot \nabla_{\mm{h}}\partial_3\sigma+{\kappa  }^{-1}(\partial_3^2 u_{\mm{h}}\cdot \nabla_{\mm{h}}(\partial_3 \sigma)^2
	+3\partial_3 u_{\mm{h}} \cdot \nabla_{\mm{h}}(\partial_3 \sigma\partial_3^2 \sigma)))\mm{d}x\nonumber \\
&-\int
 \partial_3(\partial_3\sigma\partial_3^2 u_{\mm{h}}+{\kappa  }^{-1} \partial_3 \sigma(\partial_3 \sigma\partial_3^2 u_{\mm{h}}
	+3\partial_3^2 \sigma\partial_3 u_{\mm{h}} ))\cdot \nabla_{\mm{h}}\partial_3^2\sigma\mm{d}x
   \nonumber \\ 	\lesssim&(1+\kappa  ^{-1}\| \sigma\|_3)\| \sigma\|_{1,2}\| \sigma\|_3\|\nabla u\|_{2}.
    \label{2501241000}
    \end{align}
    
Similar to $J_{2,3}$, we shall use the equation $\eqref{1.7}_1$ twice to rewrite $J_3^{2,3}$ as follows:
\begin{align}
J_3^{2,3} 
=&-{\kappa  }^{-2}\int\partial_3\sigma(\partial_3^3 \sigma)^2\partial_3(\sigma_t+u\cdot\nabla \sigma)\mm{d}x  \nonumber \\
=&{\kappa  }^{-2}\int\left( {(\partial_3\sigma)^2  } \partial_3^3\sigma \partial_3^3\sigma_t
    -\partial_3\sigma(\partial_3^3 \sigma)^2 \partial_3(u\cdot\nabla \sigma)\right)\mm{d}x-\frac{1}{2{\kappa  }^{2}}\frac{\mm{d}}{\mm{d}t}\int(\partial_3\sigma \partial_3^3 \sigma)^2\mm{d}x  \nonumber \\
=&-{\kappa  }^{-2}\int \partial_3\sigma \partial_3^3\sigma\left( 
 {\partial_3\sigma} \partial_3^3({\kappa  } u_3+u\cdot\nabla \sigma )+  \partial_3^3 \sigma \partial_3(u\cdot\nabla \sigma)\right)\mm{d}x  \nonumber \\
    &-\frac{1}{2{\kappa  }^{2}}\frac{\mm{d}}{\mm{d}t}\int(\partial_3\sigma \partial_3^3 \sigma)^2\mm{d}x
  \nonumber \\
=&
    -\kappa^{-1 }\int(\partial_3\sigma)^2\partial_3^3\sigma\partial_3^3u_3\mm{d}x
    -{\kappa  }^{-2}\int\partial_3\sigma(\partial_3^3 \sigma)^2\partial_3u\cdot\nabla \sigma\mm{d}x-{\kappa  }^{-2}\int(\partial_3\sigma)^2\partial_3^3\sigma (\partial_3^3u\cdot\nabla \sigma  \nonumber \\
    &
    +3\partial_3^2u\cdot\nabla\partial_3 \sigma
    +3\partial_3u\cdot\nabla\partial_3^2 \sigma )\mm{d}x-\frac{1}{2{\kappa  }^{2}}\frac{\mm{d}}{\mm{d}t}\int(\partial_3\sigma \partial_3^3 \sigma)^2\mm{d}x  \nonumber \\
=:& J^{2,3}_{3,1}+J^{2,3}_{3,2}+J^{2,3}_{3,3}-\frac{1}{2{\kappa  }^{2}}\frac{\mm{d}}{\mm{d}t}\int(\partial_3\sigma \partial_3^3 \sigma)^2\mm{d}x. \label{2501240957}
\end{align}
  By  the incompressible condition, the integration by parts, the product estimate,  and  the boundary condition  of $ \partial_3^2\sigma $  in \ \eqref{2.1}, we have
    \begin{align*}
    	J^{2,3}_{3,1}=& \kappa^{-1 }\int(\partial_3\sigma)^2\partial_3^3\sigma\partial_3^2\mm{div}_{\rm h} u_{\rm h}\mm{d}x\\
    	=&\frac{1}{2\kappa }\int \left(\nabla_{\rm h}\partial_3^2\sigma\cdot \partial_3((\partial_3\sigma)^2\partial_3^2u_{\rm h})  - \partial_3^3\sigma\partial_3^2u_{\rm h}\cdot\nabla_{\rm h}(\partial_3\sigma)^2\right)\mm{d}x\\ 
    	\lesssim&\kappa  ^{-1}\| \sigma\|_{1,2}\| \sigma\|_3^2\|\nabla u\|_2.
    \end{align*} 
In addition, we can use   the product estimate and  the anisotropic interpolation inequality \eqref{lemma4.7.2} to obtain 
    \begin{align*}
    	J^{2,3}_{3,2}\lesssim&\kappa  ^{-2}\|\nabla \sigma\|_{L^\infty}^2\|\partial_3^3\sigma\|_0^2\|\partial_3u\|_2\\
    	\lesssim&\kappa  ^{-2} \|\partial_3^3\sigma\|_0^2\|\nabla \nabla_{\rm h}\sigma\|_{1} \|\nabla \sigma\|_{2} \|\partial_3u\|_{2} 
    	\lesssim\kappa  ^{-2}\| \sigma\|_{1,2}\| \sigma\|_3^3\|\nabla u\|_2.
    \end{align*}
    Similarly, we also get
    \begin{align*}
    	J^{2,3}_{3,3} 
    	\lesssim&\kappa  ^{-2}\|\nabla_{\rm h}\partial_3\sigma\|_1 \|\partial_3\sigma\|_{2}^2(
\|\nabla\partial_3^2\sigma\|_0    \|\partial_3u\|_{2}+\|\nabla\partial_3\sigma\|_{1}\|\partial_3^2u\|_1+
\|\nabla \sigma\|_{2}\|\partial_3^3u\|_{0})\\
    	\lesssim&\kappa  ^{-2}\| \sigma\|_{1,2}\| \sigma\|_3^3\|\nabla u\|_2.
   \end{align*}
Thus substituting the above three estimates into \eqref{2501240957} yields 
	\begin{align}\label{2.saf29}
	J^{2,3}_3\leqslant\| \sigma\|_{1,2}(\kappa  ^{-1}\| \sigma\|_3^2+\kappa  ^{-2}\| \sigma\|_3^3)\|\nabla u\|_2-\frac{1}{2{\kappa  }^{2}}\frac{\mm{d}}{\mm{d}t}\int(\partial_3\sigma \partial_3^3 \sigma)^2\mm{d}x.
	\end{align}

Now inserting \eqref{20250124}, \eqref{2501241000} and \eqref{2.saf29} into \eqref{202412171632} and then using Young's inequality, we obtain 
	\begin{align}\label{2.29}
	J_{2,3}\leqslant {3}\frac{\mm{d}}{\mm{d}t}\int\left({\kappa  ^{-1}}-\frac{7}{2{\kappa  ^{2}}}\partial_3\sigma\right)\partial_3\sigma(\partial_3^3 \sigma)^2\mm{d}x
+c\| \sigma\|_{1,2}(\| \sigma\|_3+\kappa  ^{-2}\| \sigma\|_3^3)\|\nabla u\|_2.
	\end{align}

Thanks to the three estimates \eqref{2.27}, \eqref{2.28} and \eqref{2.29},  we derive from \eqref{2337} that
	\begin{align}\label{2.30}
		J_{2}\leqslant  {3}\frac{\mm{d}}{\mm{d}t}\int\left({\kappa  ^{-1}} -\frac{7}{2{\kappa  ^{2}} }\partial_3\sigma\right)\partial_3\sigma(\partial_3^3 \sigma)^2\mm{d}x
+c\| \sigma\|_{1,2}(\| \sigma\|_3+\kappa  ^{-2}\| \sigma\|_3^3)\|\nabla u\|_2.
	\end{align}
Finally, putting \eqref{2.26} and \eqref{2.30} into \eqref{2.21}, we arrive at the desired estimate  \eqref{2.17}. 

 (2) 
Applying $\partial_{t}$ to \eqref{1.7}$_{1}$ and \eqref{1.7}$_{2}$, we get
\begin{equation}\label{2.37}
\begin{cases}
\partial_t(\sigma_{t}+u\cdot\nabla \sigma+ {\kappa  }u_3)=0,\\
\partial_t({\rho}(u_{t}+ u\cdot\nabla u)+\nabla\beta -\mu\Delta u+ \Delta{\sigma}(\kappa  \mathbf{e}^3+ \nabla{\sigma}))=0.
\end{cases}
\end{equation} Taking the inner products of \eqref{2.37}$_1$ resp. \eqref{2.37}$_2$ and $- \Delta \sigma_{t}$, resp. $u_{t}$ in $L^2$,  then following the argument of \eqref{2.16}, and finally making use of the product estimates, the relation $\rho_t=\kappa^{-1}\sigma_t$  and \eqref{2.44}, we have
\begin{align*}
&\frac{1}{2}\frac{\mm{d}}{\mm{d}t}(\| \nabla \sigma_t\|^2_0+\|\sqrt{\rho}u_t\|^2_0)+\mu\|\nabla u_t\|^2_0\nonumber\\
=&\int\left(\Delta{\sigma_t}u\cdot\nabla{\sigma_t}-(\rho_tu_t+\partial_t({\rho} u)\cdot\nabla u+ \Delta{\sigma}\nabla{\sigma_t})\cdot u_t\right)\mm{d}x\nonumber\\
\lesssim&\| \sigma_t\|_2\left(\| \sigma_t\|_1\|u\|_2+(\| \sigma\|_3+\kappa  ^{-1} \|u\|_2^2 )\|u_t\|_0\right)\nonumber\\
&+((1+\kappa  ^{-1}\| \sigma\|_2)\|u\|_2+\kappa^{-1}\|\sigma_t\|_1)\|u_t\|_0\|u_t\|_1,
\end{align*}
which yields \eqref{2.42}. This completes the proof. 
\end{proof}

Next we shall derive the dissipative  estimates for $\Delta\sigma_t$ and $u_t$.
\begin{lem}\label{lem2.4}
It holds that
\begin{align}\label{2.31}
&\frac{\mm{d}}{\mm{d}t}\left(\frac{\mu}{2}\|\nabla u\|^2_0
- \int\Delta \sigma \sigma_t\mm{d}x\right)
+\|\sqrt{\rho}u_t\|^2_0\nonumber\\
\lesssim&\| \nabla \sigma_t\|_0^2
+\|\sigma_t\|_1(\|\sigma\|_{1,2}\|u\|_1+\|\sigma\|_3\|\nabla u\|_1)+ (1+\kappa  ^{-1}\|\sigma\|_2)\|\nabla u\|_0\| u\|_{2}\|u_t\|_0
\end{align}
and
\begin{align}\label{2.32}
&\frac{\mm{d}}{\mm{d}t}\left(\frac{\mu}{2}\|\Delta u\|_0^2
-\int\rho \Delta u\cdot u_t\mm{d}x\right)
+\|\Delta \sigma_t\|_0^2\nonumber\\
\lesssim&\kappa  ^{-1}\| \sigma_t\|_0\|u\|_2^2\|\nabla u\|_2
+(1+\kappa  ^{-1}\| \sigma\|_2)\|u_t\|_1(\| \nabla u_t\|_0+\|u\|_2\|\nabla u\|_2)\nonumber\\
&+\| \sigma_t\|_2(\| \sigma\|_{1,2}\|u\|_2+\| \sigma\|_3\|\nabla u\|_2).
\end{align}
\end{lem}
\begin{proof}
(1) 
Taking the inner product of \eqref{1.7}$_{2}$ and $u_t$ in $L^2$, and then using the integration by parts and the boundary condition \eqref{20220202081737}, we get
\begin{align}
&\frac{\mu}{2}\frac{\mm{d}}{\mm{d}t}\|\nabla u\|^2_0
+\|\sqrt{\rho}u_t\|^2_0=-\int( \Delta \sigma (\kappa  \mathbf{e}^3
+\nabla{\sigma})+{\rho} u\cdot\nabla u)\cdot u_t\mm{d}x.
\label{2501222339}
\end{align}
Using $\eqref{1.7}_{1}$ and the boundary condition of $\sigma$ in \eqref{2.1}, we find that
\begin{align*}
	-\kappa  \int \Delta \sigma\partial_tu_3\mm{d}x=&  \int\Delta \sigma\partial_t(\sigma_t+u\cdot\nabla \sigma)\mm{d}x\\
	=&  \frac{\mm{d}}{\mm{d}t}\int\Delta \sigma \sigma_t\mm{d}x+\| \nabla \sigma_t\|_0^2+  \int\Delta \sigma\partial_t(u\cdot\nabla \sigma)\mm{d}x.
\end{align*}Putting the above identity into \eqref{2501222339} yields 
\begin{align}\label{2.33}
&\frac{\mm{d}}{\mm{d}t}\left(\frac{\mu}{2}\|\nabla u\|^2_0
-  \int \Delta \sigma\sigma_t\mm{d}x\right)
+\|\sqrt{\rho}u_t\|^2_0\nonumber\\
=&\| \nabla \sigma_t\|_0^2
+\int( \Delta{\sigma} u \cdot\nabla{\sigma_t} 
- {\rho} u\cdot\nabla u\cdot u_t)\mm{d}x.
\end{align}

Exploiting  H\"older inequality,  the incompressible condition,  the integration by parts, the product estimate, the boundary condition of $u_3$ in \eqref{20220202081737}, \eqref{2.44}, Poincar\'e-type inequality \eqref{3}, we obtain
\begin{align} 
	\int\Delta{\sigma} u\cdot\nabla{\sigma_t}\mm{d}x 
	=&-\int (u_{\rm h}\cdot\nabla_{\rm h}\Delta{\sigma} + u_3\partial_3\Delta{\sigma}){\sigma}_t\mm{d}x\nonumber\\
	\lesssim&\|\sigma_t\|_1(\|\sigma\|_{1,2}\|u\|_1+\|\sigma\|_3\|\nabla u\|_1).\nonumber
\end{align}
Similarly, we also have
\begin{align} 
	-\int{\rho} u\cdot\nabla u\cdot u_t\mm{d}x\lesssim (1+\kappa  ^{-1}\|\sigma\|_2)\|\nabla u\|_0\| u\|_{2}\|u_t\|_0.\nonumber
\end{align}
Putting the above two estimates  into  \eqref{2.33} yields \eqref{2.31}.

(2) Applying $\Delta$ to the mass equation \eqref{1.7}$_1$ yields
$$\Delta (\sigma_t+u\cdot\nabla \sigma+ {\kappa  } u_3)=0.$$
Taking the inner product of the above identity and $   \Delta \sigma_t$ in $L^2$, we have
\begin{align} 
\| \Delta \sigma_t\|_0^2=-\int\Delta \sigma_t\Delta ( \kappa   u_3
+  u\cdot\nabla \sigma )\mm{d}x.\nonumber
\end{align}
Taking the inner product of \eqref{2.37}$_2$ and $-\Delta u$ in $L^2$, and then using the integration by parts, the boundary conditions of $(u_3,\partial_3^2u_3)$ in \eqref{20220202081737}, \eqref{20225012412348}, we arrive at \begin{align} 
&\frac{\mm{d}}{\mm{d}t}\left(\frac{\mu}{2}\|\Delta u\|_0^2
-\int\rho \Delta u\cdot u_t\mm{d}x\right)\nonumber \\
=&\int(\partial_t(\Delta \sigma(\kappa   \mathbf{e}^3+  \nabla \sigma)
+ \rho u\cdot\nabla u )\cdot\Delta u-\rho u_t\cdot\Delta u_t)\mm{d}x.\nonumber
\end{align}
Summing up the above  two identities, we obtain
\begin{align}\label{2.39}
&\frac{\mm{d}}{\mm{d}t}\left(\frac{\mu}{2}\|\Delta u\|_0^2
-\int\rho \Delta u\cdot u_t\mm{d}x\right)
+\| \Delta \sigma_t\|_0^2\nonumber\\
=& \int\left(( \Delta \sigma\nabla \sigma_t+\Delta \sigma_t\nabla \sigma+ \partial_t(\rho u)\cdot\nabla u+\rho u\cdot\nabla u_t
 )\cdot\Delta u
\right. 
\nonumber\\
&\left.  - \Delta \sigma_t \Delta(u\cdot\nabla \sigma) - \rho u_t\cdot\Delta u_t  \right)\mm{d}x
=J_3+J_4,
\end{align}
where we have defined that
\begin{align*}
	J_3:=&\int\Big{((}\Delta \sigma \nabla \sigma_t+\Delta \sigma_t \nabla \sigma+(\rho u_t+\rho_tu)\cdot\nabla u
	)\cdot\Delta u\\
	&\ \ \ \ \ \ - \Delta \sigma_t(2\partial_iu\cdot\nabla\partial_i \sigma+\Delta u\cdot\nabla \sigma )\Big{)}\mm{d}x,\\
	J_4:=&\int(\rho u\cdot\nabla u_t\cdot\Delta u-\rho u_t \cdot\Delta u_t- \Delta \sigma_t u\cdot\nabla\Delta \sigma)\mm{d}x.
\end{align*}
	
Making use of the product estimates, \eqref{2.44} and the relation $\rho_t=\kappa^{-1}\sigma_t$, it is easy to have
\begin{align}\label{2.40} 
J_3\lesssim(\| \sigma\|_2\| \sigma_t\|_2 +(\kappa  ^{-1}\| \sigma_t\|_0\|u\|_2 +(1+\kappa  ^{-1}\| \sigma\|_2)\|u_t\|_0)\|u\|_2)\|\nabla u\|_2.
\end{align}
Exploiting  the integration by parts, the product estimate, the boundary condition \eqref{20220202081737}, and Poincar\'e-type  inequality \eqref{3}, we arrive at
\begin{align}\label{2.41}
	J_4=&\int( \partial_i(\rho u_t)\cdot \partial_i u_t+\rho u \cdot \nabla u_t\cdot \Delta u -  \Delta \sigma_t(u_{\rm h}\cdot\nabla_{\rm h}\Delta \sigma+u_3\partial_3\Delta \sigma ))\mm{d}x\nonumber\\
	\lesssim&(1+\kappa  ^{-1}\| \sigma\|_2)\|u_t\|_1(\|\nabla u_t\|_0+\|\nabla u\|_1\|u\|_2)\nonumber\\
&+\| \sigma_t\|_2(\| \sigma\|_{1,2}\|u\|_2+\| \sigma\|_3\|\nabla u\|_2).
\end{align}
Consequently, inserting \eqref{2.40} and \eqref{2.41} into \eqref{2.39} yields \eqref{2.32}.  
\end{proof}

Next we shall establish the energy estimate of
$\kappa \|(\nabla_{\mm{h}}\partial_{3}\sigma ,-\Delta_{\mm{h}} \sigma) \|_0$ and the dissipation estimate
of $\| \sigma\|_{1,2} $.
\begin{lem}\label{lem2.5sd}
	It holds that
\begin{align}
&\kappa \|(\Delta_{\mm{h}} \sigma,\nabla_{\mm{h}}\partial_{3}\sigma   ) \|_0
\nonumber 
\\ 
&\lesssim \|u\|_2+ (1+\kappa^{-1}\|\sigma\|_2)(\|u_t\|_0+\|\nabla u\|_0 \|u\|_2)+ \| \sigma\|_2  \| \sigma\|_3.\label{202411281534} \end{align}
and \begin{align}
\label{2.43}
\| \sigma\|_{1,2} \lesssim& \kappa  ^{-1}(\|\nabla \Delta u\|_0+\| \sigma\|_{1,2} \| \sigma\|_3)+ (\kappa  ^{-1}+\kappa  ^{-2}\| \sigma\|_2) 
\left(\|u_t\|_{1} +\|u\|_{2} \|\nabla u\|_1\right) 
\end{align}
\end{lem}
\begin{proof}
(1)   We can rewrite  $(\ref{1.7})_2$ as follows
\begin{align}\label{3.3} \kappa  (\partial_{1}\partial_{3}\sigma,\partial_{2}\partial_{3}\sigma,-\Delta_{\mm{h}} \sigma)^\top =
\nabla\mathcal{P} -\mu\Delta u+\rho (u_t+ u\cdot\nabla u)+  \nabla \sigma \Delta \sigma
 ,\end{align}
where $\mathcal{P}=\beta +\kappa   \partial_3\sigma$.
  Taking the inner product of the above identity and   $ (\partial_{1}\partial_{3}\sigma,\partial_{2}\partial_{3}\sigma,-\Delta_{\mm{h}} \sigma)^\top $  in $L^2$, and then using the integration by parts and the boundary conditions of $(\sigma,\partial_3\rho)$ in \eqref{2.1},  we have
\begin{align*}
&\kappa  \int|\partial_{1}\partial_{3}\sigma,\partial_{2}\partial_{3}\sigma,-\Delta_{\mm{h}} \sigma|^2{\rm d}x\\
=&  
 \int(   \nabla \sigma\Delta \sigma-\mu\Delta u+\rho( u_t+ u\cdot\nabla u))
\cdot(\partial_{1}\partial_{3}\sigma,\partial_{2}\partial_{3}\sigma,-\Delta_{\mm{h}} \sigma)^\top {\rm d}x .\end{align*}
We  easily deduce from the above identity that
\begin{align}
&\kappa   \|(\nabla_{\mm{h}} \partial_{3}\sigma,-\Delta_{\mm{h}} \sigma) \|_0
\nonumber 
\\ 
&\lesssim \|u\|_2+(1+\kappa^{-1}\|\sigma\|_2)(\|u_t\|_0+\|\nabla u\|_0 \|u\|_2)+\| \sigma\|_2 \| \sigma\|_3 , \nonumber  \end{align}
 which yields \eqref{202411281534}.
 
(2)  Applying $\mm{curl}$ to the momentum equation \eqref{1.7}$_2$, we can obtain the vortex equation 
\begin{align}\label{2.8}
&\kappa   \nabla^{\bot}\Delta \sigma=\rho(\omega_t+u\cdot\nabla\omega)-
 \mu\Delta\omega+{\mathbf{M}}+\mathbf{N},
\end{align}
where we have defined that
\begin{align*}
\begin{cases}
\mathbf{M}:=(-\partial_{t}u_2,\partial_{t}u_1,0)^{\top},\ \mathbf{N}:= \mathbf{N}^{\mm{m}}+\mathbf{N}^{\mm{c}}+  \mathbf{N}^{\mm{k}},\\
\mathbf{N}^{\mm{m}}:= (\partial_{2}\rho\partial_{t}u_3-\partial_{3}\rho\partial_{t}u_2,\partial_{3}\rho\partial_{t}u_1
-\partial_{1}\rho\partial_{t}u_3,\partial_{1}\rho\partial_{t}u_2-\partial_{2}\rho\partial_{t}u_1)^{\top},\\
\mathbf{N}^{\mm{c}}:=\big(\partial_{2}({\rho}u)\cdot\nabla u_3-\partial_{3}({\rho}u)\cdot\nabla u_2,
\partial_{3}({\rho}u)\cdot\nabla u_1-\partial_{1}({\rho}u)\cdot\nabla u_3,\\
\qquad\quad\partial_{1}({\rho}u)\cdot\nabla u_2-\partial_{2}({\rho}u)\cdot\nabla u_1\big)^{\top},\\ 	
\mathbf{N}^{\mm{k}}:= (\partial_{3}\sigma\partial_{2}\Delta \sigma-\partial_{2}\sigma\partial_{3}\Delta \sigma,\partial_{1}\sigma\partial_{3}\Delta \sigma
-\partial_{3}\sigma\partial_{1}\Delta \sigma,\partial_{2}\sigma\partial_{1}\Delta \sigma-\partial_{1}\sigma\partial_{2}\Delta \sigma)^{\top}.
\end{cases}	
\end{align*} 
Exploiting the definition of $\rho$ and the product estimates, we have 
\begin{align}
\|({\mathbf{M}}_{\mm{h}}, {\mathbf{N}}_{\mm{h}})\|_0
\lesssim&(1+\kappa  ^{-1}\| \sigma\|_2)\left(\|u_t\|_{1}+\|\nabla u\|_1\|  u\|_2\right)+\| \sigma\|_{1,2} \| \sigma\|_3
\label{2501241402}
\end{align}

Taking the inner product of the vortex equation \eqref{2.8} and $-  \nabla^{\bot}\Delta \sigma$ in $L^2$, and then using  H\"older inequality, the product estimate, \eqref{2.44} and \eqref{2501241402}, we get that
\begin{align*}
&\kappa  \| \Delta(\partial_2\sigma,\partial_1\sigma)\|_0^2\nonumber\\
\lesssim&  \int(\rho(\partial_{t}\omega _{\mm{h}}
+u\cdot\nabla\omega _{\mm{h}})-\mu{\Delta\omega_{\mm{h}}}+{\mathbf{M}}_{\mm{h}}+ {\mathbf{N}}_{\mm{h}})\cdot\nabla_{\mm{h}}^{\bot}\Delta \sigma\mm{d}x\nonumber\\
\lesssim& \| \sigma\|_{1,2}^2\| \sigma\|_3+\| \sigma\|_{1,2}\left(\|\nabla\Delta  u\|_0+(1+\kappa  ^{-1}\| \sigma\|_2)(\|u_t\|_{1}+ \|u\|_2 \|\nabla u\|_1)\right)
,
\end{align*}	
which, together with \eqref{2.46} with $s=1$ and $2$, yields \eqref{2.43}. 
This completes the proof.
\end{proof}
 
Finally we shall derive that  $\| \sigma_t\|_1$ and $\|u_t\|_{0}$ can be controlled by the norms of spatial derivatives of $(\sigma,u)$. 
\begin{lem}\label{lem2.7}
It holds that
\begin{align}
\| \sigma_t\|_1 \lesssim  \kappa  \|u_3 \|_{1}+\| \nabla \sigma \|_{1}\|u \|_2 \label{2saf.47}
\end{align}
and
\begin{align}
 \|u_t\|_{0}\lesssim &{\kappa }  \|(\Delta_{\mm{h}} \sigma ,\nabla_{\mm{h}} \partial_3\sigma )\|_{0}+(1+(1+\kappa^{-1}\|\sigma\|_2)\|u \|_1)\|u \|_2+ \| \nabla \sigma \|_{1} \| \nabla \sigma \|_{2} .\label{2.47}
\end{align}
\end{lem}
\begin{proof} 
It is easy see from the product estimate and 
 $(\ref{1.7})_1$ that
\begin{align*}
\| \sigma_t \|_{1}\lesssim &\kappa  \|u_3 \|_{1}  + \|u \cdot\nabla \sigma \|_{1} 
\lesssim \kappa  \|u_3 \|_{1}+\| \nabla \sigma \|_{1} \|u \|_2, 
\end{align*}
which  yields \eqref{2saf.47}.

 It is well-known from the incompressible condition and the mass equation \eqref{1.1}$_1$ that
\begin{equation}
0<d\leqslant  \inf\limits_{x\in\Omega}\{\bar{\rho}(x)+\kappa  ^{-1 }\sigma^0\}\leqslant{\rho}(t,x)\mbox{ for any }(x,t)\in\Omega\times I_T, \label{2.441}
\end{equation} 
where $\rho=\bar{\rho} +\kappa  ^{-1 }\sigma$.
 Taking the inner product of \eqref{3.3} and  $u_t$ in $L^2$, and then using the integration by parts, and  the boundary condition of $u_3$ 
 in \eqref{20220202081737}, we get
\begin{align*}
\|\sqrt{\rho}u_t\|^2_0=&\int ((\mu \Delta u -    \nabla{\sigma}\Delta{\sigma} 
- {\rho} u\cdot\nabla u)\cdot u_t+\kappa   ( \partial_tu_{\mm{h}} \cdot\nabla_{\mm{h}}\partial_3 \sigma-\Delta_{\mm{h}} \sigma\partial_tu_3))\mm{d}x.
\end{align*}
Exploiting H\"older inequality, the product estimate, \eqref{2.44} with $i=0$ and \eqref{2.441}, we can get \eqref{2.47} from the above identity. This completes the proof.  
\end{proof}

\subsection{\emph{A priori} stability estimates}
Now we are in the position to building the total energy inequality \eqref{1.12} for the initial boundary value problem  \eqref{20220202081737}--\eqref{1.8x}. 
\begin{pro}[\emph{A priori} estimates]\label{lem3}
Let $(\sigma,u)$ be the solution of the initial-boundary value problem \eqref{20220202081737}--\eqref{1.8x} defined on $  \Omega\times I_T$ with the initial data
 $(\sigma^0,u^0)\in H^3_{\bar{\rho}}\times {\mathcal{H}^2_{\mathrm{s}}}$ satisfying \eqref{202412051008}.
Then there exist constants $c_1$, $c_2$ and $\chi$, where
$$c_1,\ c_2\geqslant 1\mbox{ and }\chi=  \max\{\kappa  ^{-1},c_2\} $$ 
 such that,  if $(\sigma,u)$ satisfies
\begin{align}
 \mathop{\rm sup}_{0\leqslant t< T}{ {E}}(t)\leqslant (2c_1\chi^9)^{-2},
 \label{2022412052137}
\end{align} 
  then the solution satisfies the following \emph{a priori} stability estimate
\begin{align}\label{1.saf12}
\mathop{\rm sup}_{0\leqslant t< T} \mathcal{E}(t)+\int_0^T\mathcal{D}(\tau)d\tau\leqslant 2c_1\chi^9E^0 .
            \end{align} 
\end{pro}
\begin{proof}
Let  
\begin{align}\label{1af12}
\mathop{\rm sup}_{0\leqslant t< T}{ {E}}(t)\leqslant \delta^2\leqslant 1.
 \end{align}  In view of \eqref{2saf.47} and \eqref{2.47}, it is easy to see that 
\begin{align}
\label{202412052151}
\mathcal{E}(t)\lesssim  (1+\kappa^{-1}\|\sigma\|_2)E(t)\mbox{ for any }t\in [0,T),
\end{align}
where $\delta$ will be defined by \eqref{2.511}.
In addition, we can derive from $(\ref{1.7})_1$ and \eqref{1af12} that
	\begin{align}
\kappa  \| u_3\|_{1}\lesssim  \|\sigma_t\|_{1}+\|u\cdot\nabla \sigma\|_{1} 
\lesssim  \|\sigma_t\|_{1}+\|\sigma\|_2\|u\|_2
\lesssim \|\sigma\|_2+ \|\sigma_t\|_{1}   .
\label{2501241201}
\end{align}

Exploiting Young's inequality, \eqref{2saf.47}, \eqref{1af12}, \eqref{20240129182211} with $i=1$ and \eqref{2.46} with $s=t$, we derive from Lemmas \ref{lem2.2}--\ref{lem2.4} and \eqref{2.43} that, for sufficiently large positive constant $\chi\geqslant 1$ (independent of $\kappa  $),
\begin{align}
&\chi\frac{\mm{d}}{\mm{d}t}\tilde{\mathcal{E}}(t)+c\tilde{\mathcal{D}}(t)
\lesssim\left(\chi^3\left((\chi^4(1+\kappa  ^{-1}) +\kappa  ^{-2})\right)+\kappa^{-4}\right) {E}^{1/2}(t)
\mathcal{D}(t)\nonumber\\
&\qquad\qquad \qquad \qquad +\kappa  ^{-2} (\|\nabla \Delta u\|_0^2+\|u_t\|_1^2 ),\label{2.49}
\end{align}
 where we have defined that
\begin{align*}
\tilde{\mathcal{E}}(t):=&\frac{\chi^2}{2}\left(\|( \nabla   \sigma, \nabla \Delta \sigma)\|^2+\|\sqrt{\rho}( 
 u,\Delta u)\|^2_0\right)
-3\chi^2\int\left(\kappa ^{-1}-\frac{7}{2\kappa^2}\partial_3\sigma \right)\partial_3\sigma(\partial_3^3 \sigma)^2\mm{d}x\\
&+\chi^3\left(\frac{\mu}{2}\|\nabla u\|^2_0
- \int \Delta \sigma\sigma_t\mm{d}x\right)+\chi^4\left(\frac{\mu}{2}\|\Delta u\|_0^2
-\int\rho\Delta u \cdot u_t\mm{d}x\right)+\chi^6\|( \nabla \sigma_t,\sqrt{{\rho}}u_t)\|^2_{0}),\\
\tilde{\mathcal{D}}(t):=&\| \sigma\|_{1,2}^2+\chi(\chi^2\|\nabla u\|^2_2 +\chi^4\| \sigma_t\|_2^2 +\chi^3\| u_t\|^2_1).
\end{align*}
In addition, making use of  Young's inequality, the definitions of $\mathcal{E}(t)$, $\mathcal{D}(t)$, \eqref{2.44} with $i=0$, \eqref{202411281534}, \eqref{2.441}, \eqref{1af12}, \eqref{2501241201}, 
Poincar\'e-type inequality \eqref{3}, \eqref{20240129182211} with $i=0$ and \eqref{2.52}, we have,  for sufficiently large positive constant $\chi$,
\begin{align}&{\mathcal{E}}(t)\lesssim  \tilde{\mathcal{E}}(t)+c\chi^2\left(\chi^2\kappa^{-1}+\kappa^{-2}  \right)\|\sigma\|_3{\mathcal{E}}(t) ,\label{202as7}\\
& \tilde{\mathcal{E}}(t)\lesssim\left( \chi^6(1+ \kappa^{-1})+\chi^2  \kappa^{-2}\right){\mathcal{E}}(t) ,\label{202asd412041807}
\\
&\mbox{and } {\mathcal{D}}(t)\lesssim \tilde{\mathcal{D}}(t).
\label{202412041807}
\end{align}

It is easy to see from \eqref{2.49} and \eqref{202412041807} that there exists $c_2\geqslant 1$ such that the following estimate holds for any $\chi\geqslant  \max\{\kappa  ^{-1},c_2\}$:
\begin{align}\label{2.50}
	\chi\frac{\mm{d}}{\mm{d}t}\tilde{\mathcal{E}}(t)+c{\mathcal{D}}(t) 
\lesssim \chi^8  {E}^{1/2}(t) \mathcal{D}(t).
\end{align}Integrating the above inequality over $(0,t)$ we obtain 
\begin{align*} 
\chi\tilde{\mathcal{E}}(t)+\int_0^t{\mathcal{D}}(\tau)\mm{d}\tau
\leqslant &\chi\tilde{\mathcal{E}}(t) |_{t=0}+ \chi^8 \sup_{0\leqslant \tau\leqslant t}{E}^{1/2}(\tau) \int_0^t\mathcal{D}(\tau)\mm{d}\tau,
\end{align*}
which, together with \eqref{202412052151}, \eqref{202as7} and \eqref{202asd412041807}, implies
\begin{align}\label{2.51}
{\mathcal{E}}(t)+\int_0^t{\mathcal{D}}(\tau)\mm{d}\tau
\leqslant & c_1\chi^9\left(E^0+\|\sigma\|_3\mathcal{E}(t) +\sup_{0\leqslant \tau\leqslant t}{E}^{1/2}(\tau) \int_0^t\mathcal{D}(\tau)\mm{d}\tau\right).
\end{align}
If  
\begin{align}\label{2.511}\delta:= (2c_1\chi^9)^{-1} <1,\end{align}
we can derive from \eqref{1af12} and \eqref{2.51} that
\begin{align}\label{2saf.51}
{\mathcal{E}}(t)+\int_0^t{\mathcal{D}}(\tau)\mm{d}\tau
\leqslant 2c_1\chi^9E^0 ,
\end{align}which yields \eqref{1.saf12}. This completes the proof. 
\end{proof}

\section{Proof of Theorem \ref{thm2}}

Now we introduce the local(-in-time) well-posedness of the initial-boundary value problem \eqref{20220202081737}--\eqref{1.8x} for any fixed $\kappa  $.
\begin{pro}\label{202102182115}
Let $\mu$, $\kappa  $ be given positive constants, $\bar{\rho}$ be defined by \eqref{202501232128}, and $(\sigma^0,u^0 )\in   H^3_{\bar{\rho}}\times {\mathcal{H}^2_{\mathrm{s}}}$ satisfy \eqref{202412051008} . Then there exists  $T^{\max}>0$ such that the initial-boundary value problem \eqref{20220202081737}--\eqref{1.8x} admits a unique local(-in-time) strong solution $(\varrho,v)$ defined on $\Omega\times [0,T^{\max})$ with an associated pressure $\beta$. Moreover
\begin{itemize}
\item  $(\sigma,v,\nabla\beta)\in{\mathfrak{P}} _{T}\times { \mathcal{V}_{T} }\times C^0(\overline{I_{T}}, L^2)$ for any $T \in I_{T^{\max}}$,  and \begin{equation*}
0<\inf\limits_{x\in\Omega}\big\{{\rho}^{0}(x)\big\}\leqslant {\rho}(t,x)\leqslant \sup\limits_{x\in\Omega}\big\{{\rho}^{0}(x)\big\}\mbox{ for any }(t,x)\in \Omega\times I_{T^{\max}},
\end{equation*}
where $\rho^0:=\varrho^0+\kappa  ^{-1 }\bar{\rho}$. 
  \item $\limsup_{t\to T^{\max} }\| (\nabla \sigma ,v)( t)\|_2=\infty$  if  $T^{\max}<\infty$.
\end{itemize} 
\end{pro}
\begin{pf}Since Proposition \ref{202102182115} can be easily proved by the standard iteration method as in \cite[Theorem 1]{TW2010}, we omit the trivial proof.  \hfill $\Box$ \end{pf}

Thanks to the \emph{a priori} stability estimate \eqref{1.saf12} in Proposition \ref{lem3}, we can easily establish the global solvability
in Theorem \ref{thm2} based on the local solvability
in Proposition \ref{202102182115}. Next, we briefly describe the proof for the readers' convenience.

Assume that $(\sigma^0,v^0)$ satisfies \eqref{202412051008} and \eqref{2022412051007}, where $c_1$ and $c_2$ are provided by Proposition \ref{lem3}.
In view of Proposition \ref{202102182115}, there exists a unique local strong solution $(\varrho,v, \beta)$ to  the initial-boundary value problem \eqref{20220202081737}--\eqref{1.8x}  with the maximal existence time $T^{\max}$. 

By the regularity of $(\varrho,v,\beta)$, we can verify that  $(\varrho,v)$ satisfies stability estimate \eqref{1.saf12} in Proposition \ref{lem3}, i.e.
\begin{align}
	\label{omessestsimQ}
 \sup_{0\leqslant  t<  T }{ \mathcal{E} (t)   } +\int_0^{T }  \mathcal{D} (t) \mm{d}t \leqslant  2c_1\chi^9E^0  ,
\end{align}
if
$$\sup_{0\leqslant  t<  T }E(t)\leqslant (2c_1\chi^9)^{-2}\mbox{ for }T\in I_{T^{\max}}.$$

Let
\begin{align}
T^{*}=&\sup \left\{\tau\in I_{T^{\max}}~ \left|~  E(t) \leqslant  (3c_1\chi^9)^{-2}  \mbox{ for any }
\ t\leqslant\tau \right.\right\}.\nonumber
\end{align}
Then, we easily see that the definition of $T^*$ makes sense by the fact
\begin{align}\label{2022412061417}
E^0\leqslant (3c_1 \chi^9)^{-3}< (2c_1\chi^9)^{-2} .\end{align}Thus, to show the existence of a global strong solution, it suffices to verify $T^*=\infty$. We shall prove this by contradiction below.

Assume $T^*<\infty$, then by and the definition of $T^*$ and Proposition \ref{202102182115}, we have
\begin{align}T^*\in I_{T^{\max}}. \label{20222022519850}
\end{align}
Noting that
\begin{equation*}
\sup_{0\leqslant  t<    {T^{*}}}  E(t)   \leqslant   (2c_1\chi^9)^{-2},
\end{equation*}
 then, by \eqref{omessestsimQ} with $T=T^*$ and \eqref{2022412061417}, we have
\begin{align*}
 \sup_{0\leqslant  t<  T^*}{ \mathcal{E} (t)   } +\int_0^{T^*}  \mathcal{D} (t) \mm{d}t \leqslant  2c_1\chi^9E^0 \leqslant \frac{2}{27( c_1\chi^9)^{2}}.
\end{align*}
In particular,
\begin{align}
\sup_{0\leqslant  t< {T^{*}}} E(t)  \leqslant  \frac{8}{27(2c_1\chi^9)^2}< (2c_1\chi^9)^{-2}. \label{2020103261534}
\end{align}
Making use of \eqref{20222022519850}, \eqref{2020103261534} and the strong continuity $(\varrho,v)\in C^0([0,T^{\max}), H^3\times H^2)$, we deduce that
  there is a  constant $\tilde{T}\in (T^*,T^{\max})$, such that
\begin{align}
\sup_{0\leqslant  t\leqslant  {\tilde{T} }}E(t) \leqslant  (2c_1\chi^9)^{-2}, \nonumber
\end{align}
which contradicts with the definition of $T^*$. Hence, $T^*=\infty$, which implies $T^{\max}=\infty$.
This completes the proof of the existence of a global solution. The uniqueness of the global solution is obvious due to the uniqueness result of local solutions in Proposition \ref{202102182115}.
In addition, exploiting  the product estimates  and \eqref{1.12}, we can easily derive \eqref{202412061425} from \eqref{3.3}. 

To complete the proof of Theorem \ref{thm2}, we shall verify \eqref{1sdaf2}.
To this purpose, applying $\mm{curl}$ to the momentum equation \eqref{3.3}$_1$ yields 
  \begin{align} 
 \kappa  \mm{curl}(\partial_{1}\partial_{3}\sigma,\partial_{2}\partial_{3}\sigma,-\Delta_{\mm{h}} \sigma)^\top
=\mm{curl}(\rho (u_t+u\cdot\nabla u)-\mu\Delta  u+  \nabla \sigma \Delta \sigma).
\end{align}
Applying $\|\cdot\|_{L^2( I_T ,L^2)}$ to the above identity,  and then using the product estimate and \eqref{1.12}, we can estimate that
\begin{align}
&\| \kappa  \mm{curl}(\partial_{1}\partial_{3}\sigma,\partial_{2}\partial_{3}\sigma,-\Delta_{\mm{h}} \sigma)^\top\|_{L^2( I_T ,L^2)}\nonumber \\
&\lesssim  \|\mm{curl}(\rho(u_t+ u\cdot\nabla u) -\mu\Delta u+  \nabla \sigma \Delta \sigma)\|_{L^2( I_T ,L^2)}\nonumber \\
& \lesssim \sqrt{c_1\chi^9E^0} (1+\chi+\sqrt{T}) \label{3saf.3} .
\end{align}
Noting that $\mm{div}(\partial_{1}\partial_{3}\sigma,\partial_{2}\partial_{3}\sigma,-\Delta_{\mm{h}} \sigma)^\top=0 $,
we can derive from  \eqref{1.12}, \eqref{202411281534}, \eqref{3saf.3} and Lemma \ref{A.5} that 
\begin{align}
 \| \kappa   (\partial_{1}\partial_{3}\sigma,\partial_{2}\partial_{3}\sigma,-\Delta_{\mm{h}} \sigma)\|_{L^2( I_T ,H^1)}\lesssim\sqrt{c_1\chi^9E^0} (1+\chi+\sqrt{T}) .\label{3sasaf.3}
\end{align}
Thanks to  \eqref{1.12} and \eqref{3sasaf.3}, we easily deduce from \eqref{3.3} that 
\begin{align}
 \| \nabla\mathcal{P} \|_{L^2( I_T ,H^1)}\lesssim \sqrt{c_1\chi^9E^0} (1+\chi+\sqrt{T}),\nonumber 
\end{align}  
which, together with \eqref{3sasaf.3}, yields \eqref{1sdaf2}. This completes the proof of Theorem \ref{thm2}.

\section{Proof of Theorem \ref{thm3}}\label{subsec:04}
This section is devoted to the proof of the asymptotic behavior of solutions stated  in Theorem \ref{thm3}. 
For the simplicity, we denote $(\varrho^\kappa,\sigma^\kappa, u^\kappa  ) $ by $(\varrho ,\sigma, u)$, where $\varrho^\kappa=\kappa^{-1}\sigma^\kappa$.

Let $T>0$ and $p\geqslant 1$ be arbitrary given. 
 Making use of \eqref{1.12}--\eqref{1sdaf2}, Aubin--Lions theorem (see Theorem \ref{A.7}), 
Arzel\'a--Ascoli theorem (see Theorem \ref{A.10}) and Banach--Alaoglu theorem (see Theorem \ref{asdA.7}), there exits a sequence (not relabeled) such that, for $\kappa  \to\infty$, 
\begin{align}
&u\rightarrow\tilde{u}\mbox{ weakly* in } L^{\infty}(I_T,{\mathcal{H}^2_{\mathrm{s}}})\mbox{ and weakly in } L^2(I_T,H^3) \mbox{ with }\tilde{u}_3=0,\nonumber \\
&u\rightarrow{\tilde{u}}\mbox{ strongly in }L^{p}(I_T,H^2_{\mm{loc}})\cap C^0(\overline{I_T},H^1_{\mm{loc}})\mbox{ with }\tilde{u}_{\mm{h}}|_{t=0}=w^0_{\mm{h}},\label{3.9}\\ 
& \sigma\rightarrow\tilde{\varpi}\mbox{ weakly* in } L^{\infty}(I_T,H^3_0)\mbox{ and }\mbox{ strongly  in } C^0(\overline{I_T},H^2_{\mm{loc}}),\label{3.11}\\
& u_t\rightarrow \tilde{u}_t\mbox{ weakly* in } L^{\infty}(I_T,L^{2})\mbox{ and weakly in } L^2(I_T,\mathcal{H}^1), 
\label{202411261907}
\\
&(\varrho,\varrho_t)= \kappa^{-1}(\sigma,\sigma_t)\rightarrow (0,0)\mbox{ strongly in }L^{\infty}(I_T,H^3)\times L^2(I_T,{ {H}^2 }),\nonumber  \\  
&\kappa   (\nabla_{\mm{h}}\partial_3\sigma^\top,-\Delta_{\rm h} \sigma)^\top\rightarrow  \mathcal{N}  \mbox{ weakly* in } L^{\infty}(I_T,L^2)\mbox{ and weakly in }L^2( I_T , H^1_0),\label{3.6}\\
& (\nabla_{\mm{h}}\partial_3\sigma,-\Delta_{\rm h} \sigma)^\top\rightarrow 0\mbox{ strongly in } L^{\infty}(I_T,L^2),\label{3x6}\\
&\nabla\mathcal{P}\rightarrow \tilde{N}\mbox{ weakly* in } L^{\infty}(I_T,L^{2})\mbox{ and weakly in } L^2(I_T,H^1)\label{3.7};
\end{align}
moreover, for any $\chi \in C^\infty_0(I_T)$, for any $\eta \in C^\infty_0(0,h)$, for any  $\phi\in C^\infty_0(\Omega)$, for any $(\psi_1,\psi_2)^\top \in C^\infty_0(\Omega)$ satisfying $\partial_1\psi_1+\partial_2\psi_2=0$ and for any $\varphi\in C^\infty_0(\Omega)$ satisfying $\mm{div} \varphi=0$,
\begin{align}
&\int_0^T\int (\mathcal{N}_{\rm h}^\top,\mathcal{N}_3)^\top\cdot \nabla\phi\mm{d}x\chi\mm{d}\tau=0,\label{3.61}\\
&\int_0^T\int_0^h\int_{\mathbb{R}^2} (\mathcal{N}_{1} \psi_1+\mathcal{N}_{2} \psi_2)\mm{d}x_{\mm{h}}\eta\mm{d}x_3\chi\mm{d}\tau=0,  \\
&\int_0^T \int \tilde{N}\cdot\varphi\mm{d}x\chi\mm{d}\tau=0  .  \label{3.611}
\end{align} 

We can further derive from \eqref{3.61}--\eqref{3.611} that, for a.e. $t\in  I_T $,
\begin{align}
&\mm{div}(\mathcal{N}_{\rm h}^\top,\mathcal{N}_3)^\top=0  , \label{202411281345} \\
&\int_{\mathbb{R}^2} \mathcal{N}_{\mm{h}}\cdot  \psi \mm{d}x_{\mathrm{h}}=0\mbox{ for a.e. } x_3\in (0,h),  \label{3.7saf} \\
& \int \tilde{N}\cdot\varphi\mm{d}x=0 \label{3.7x}.
\end{align}
In view of  \eqref{3.7saf},  \eqref{3.7x} and Lemma \ref{A.8}, we have 
\begin{align} 
&  \tilde{N}=\nabla \mathcal{Q}\mbox{ for some }\mathcal{Q}\in L^2_{\mm{loc}} ,\label{202411126sd1859}\\
& \mathcal{N}_{\mathrm{h}}=\nabla_{\mathrm{h}} \mathcal{M}\mbox{ for some }\mathcal{M}\in L^2_{\mm{loc}}(\mathbb{R}^2)\mbox{ for a.e. } x_3\in (0,h) ,\label{2024111261859}
\end{align}
which, together with \eqref{202411281345}, yields 
\begin{align}
-\Delta_{\mm{h}}\mathcal{M}=\partial_3 \mathcal{N}_3.
\label{2022411281348}
\end{align} 
 
 In addition, it is easy to see from \eqref{3.11} and \eqref{3x6} that
 \begin{align}\label{3.20}
\Delta_{\mm{h}}\tilde{\varpi}=0 ,
\end{align}
which, together with the regularity of $\tilde{\varpi}\in L^\infty(I_T,H^3)$, imply
\begin{align} 
\tilde{\varpi}=0. 
\label{2501252312}
\end{align}  
 
By virtue of the product estimate and \eqref{1.12}, it is easily to  see that 
$$\| \sigma u_t\|_{L^{\infty}(I_T,L^{2})}+\| \sigma u\cdot\nabla u\|_{L^{\infty}(I_T,H^{1})} \lesssim 1.$$
Thus it holds that 
\begin{align}
&\kappa  ^{-1 }\sigma u_t \to 0 \mbox{ strongly in } L^{\infty}(I_T,L^{2}), \label{202411261110} \\
&\kappa  ^{-1 }\sigma u\cdot\nabla u \to 0 \mbox{ stronly in } L^{\infty}(I_T,H^{1}). \label{2024112611101}
\end{align}
Making use of \eqref{3.9}, \eqref{3.11},  \eqref{202411261907},  \eqref{2501252312}--\eqref{2024112611101} and 
\eqref{11x}, one has
\begin{align}& {\rho}u_t\to \bar{\rho}\tilde{u}_t\mbox{ weakly in } L^{\infty}(I_T,L^{2}), \label{2026fsda111}\\& {\rho} {u}\cdot\nabla{u}\to\bar{\rho}\tilde{u}\cdot\nabla\tilde{u}\mbox{ strongly in } L^{p}(I_T,L^{3}_{\mm{loc}}), \label{2026111}\\
&  \nabla \sigma\Delta \sigma\to  0\mbox{ strongly in } C^0(\overline{I_T},L^{3/2}_{\mm{loc}}). \label{2022411126111}
 \end{align}

 Let $\tilde{\beta }=\mathcal{Q}-\mathcal{M}$.
Making use of \eqref{3.9}, \eqref{3.6}, 
\eqref{3.7}, \eqref{202411126sd1859},  \eqref{2024111261859} and  \eqref{2026fsda111} --\eqref{2022411126111}, we easily derive from \eqref{1.7} with $(\nabla\partial_3\sigma-(\nabla_{\mm{h}}\partial_3\sigma^\top,-\Delta_{\mm{h}} \sigma))^\top$ in place of $\Delta \sigma \mathbf{e}^3$ that 
\begin{align}\label{3.22}
\begin{cases}
 \bar{\rho}(\partial_t\tilde{u}_{\mathrm{h}}+\tilde{u}_{\mm{h}}\cdot\nabla\tilde{u}_{\mm{h}})+
 \nabla_{\mm{h}}\tilde{\beta }-\mu\Delta\tilde{u}_{\mm{h}}
  =0 ,\\
 \mm{div}_{\mm{h}}\tilde{u}_{\mm{h}}=0 
\end{cases}
\end{align} 
and
\begin{align}
\label{202411281349}
\mathcal{N}_3 =\partial_3\mathcal{Q},\end{align} 
which, together with \eqref{2022411281348}, yields 
\begin{align}
-\Delta_{\mm{h}}\mathcal{M}=\partial_3^2 \mathcal{Q}.
\label{202241128134saf8}
\end{align} 
In addition, applying $\mm{div}_{\mm{h}}$ to \eqref{3.22}$_1$, and then using \eqref{3.22}$_2$, we arrive at 
\begin{align} 
- \Delta \mathcal{Q}=
\bar{\rho}\nabla {\tilde{u}}_{\mm{h}}:\nabla_{\mm{h}}u_{\mm{h}}.
\end{align}

Now let $\bar{u}_{\mm{h}}$, $\bar{\mathcal{Q}}$, $\bar{\mathcal{M}}$ and $\bar{\mathcal{N}}_3$ enjoy the same regularity as well as    $\tilde{u}_{\mm{h}}$, ${\mathcal{Q}}$, ${\mathcal{M}}$ and ${\mathcal{N}}_3$, and satisfy the following initial-boundary value problem
\begin{align}
&\begin{cases}
\bar{\rho}(\partial_t{\bar{u}}_{\mm{h}}+{\bar{u}}_{\mm{h}}\cdot\nabla_{\mm{h}} {\bar{u}}_{\mm{h}})+\nabla_{\rm h}(\bar{\mathcal{Q}}-\bar{\mathcal{M}})-\mu\Delta{\bar{u}}_{\mm{h}}
=0,\\ 
\mm{div}_{\mm{h}}{\bar{u}}_{\mm{h}}=0, \\
\bar{u}_{\mm{h}}|_{t=0}=w_{\mm{h}}^0,\\
- \Delta\bar{\mathcal{Q}}=
\bar{\rho}\nabla {\bar{u}}_{\mm{h}}:\nabla_{\mm{h}}\bar{u}_{\mm{h}},\ \partial_3\bar{\mathcal{Q}}|_{\partial\Omega}=0,\\
-\Delta_{\mm{h}} \bar{\mathcal{M}}=\partial_3^2\bar{\mathcal{Q}},\ 
\bar{\mathcal{N}}_3 =\partial_3\bar{\mathcal{Q}}, 
\end{cases}
\end{align} 
then it is easy to check that 
 $$(\bar{u}_{\mm{h}}, \nabla  \bar{\mathcal{Q}}, \nabla_{\mm{h}} \bar{\mathcal{M}},\bar{\mathcal{N}}_3)
 =(\tilde{u}_{\mm{h}}, \nabla  {\mathcal{Q}}, \nabla_{\mm{h}} {\mathcal{M}},{\mathcal{N}}_3).$$ 
The uniqueness mentioned above means  that any sequence of 
$$\{(u,\nabla \mathcal{P},\kappa  \partial_1\partial_3\sigma,\kappa   \partial_2\partial_3\sigma, -\kappa  \Delta_{\rm h} \sigma) \}_{\kappa  > 0}$$ converges to the  limit function  $({\tilde{u}}, \nabla  \mathcal{Q}, \partial_1  \mathcal{M}, \partial_2  \mathcal{M},\mathcal{N}_3)$ is independent of choosing the sequences of solutions.  This completes the proof of Theorem \ref{thm3}.

\appendix
\section{Analysis tools}\label{sec:04}
\renewcommand\thesection{A}
This appendix is devoted to providing some mathematical results, which have been used in previous sections. In addition, $a\lesssim b$ still denotes $a\leqslant cb$ where the positive constant $c$ depends on the parameters and the domain in the lemmas in which $c$ appears.

\begin{lem}\label{A.1}
Embedding inequality \cite[Theorem 4.12]{Adams1975}: Let $D\subset \mathbb{R}^3$ be a domain satisfying the cone condition. It holds that 
\begin{align}\label{1}
\|f\|_{C^0(\overline{D})} \lesssim\|f\|_{H^{2}(D)} 
\end{align} 
for any  $f\in H^2(D)$   (after possibly being redefined on a set of measure zero), and 
\begin{align}\label{11x}
\|\phi\|_{L^6(D)} \lesssim\|\phi\|_{H^{1}(D)}\mbox{  for any }\phi\in H^1(D).\end{align}
\end{lem}
\begin{lem}\label{A.2}
Poincar\'e-type inequality  \cite[Lemma A.10]{GT2013}: For any $f\in H_0^1$, it holds that
\begin{align}\label{3}
\|f\|_0\lesssim\|\partial_3f\|_0.
\end{align} 
\end{lem}

\begin{lem}\label{lem4.7}\begin{enumerate}[(1)]
  \item 
Interpolation inequality in $H^j$: Let $D\subset\mathbb{R}^3$ be a domain satisfying the cone condition and  $0\leqslant j<i$, then  it holds that
\begin{align}\label{8}
\|f\|_{j}\lesssim\|f\|_{0}^{1-\frac{j}{i}}\|f\|_{i}^{\frac{j}{i}}\lesssim\varepsilon^{-\frac{j}{i-j}}\|f\|_{0}
+\varepsilon\|f\|_{i}\end{align}
for any  $\varepsilon>0$ and   for any $f\in H^j(D)$.
  \item Interpolation inequalities in $L^p$:
 it holds that
	\begin{align}\label{lemma4.7.1}
	&\| f\|_{L^4}^2\lesssim\|\nabla_{\mm{h}} f\|_{0}\| f\|_{1}\mbox{ for any }f\in H^1,\\
	&\|f\|_{L^\infty}^2\lesssim\|\nabla_{\mm{h}} f\|_{1}\| f\|_{2}\mbox{ for any }f\in H^2\label{lemma4.7.2}.
	\end{align}
\end{enumerate}
\end{lem}
\begin{pf} The result \eqref{8} can be founded in  \cite[Theorem 5.2]{Adams1975}. Next we derive \eqref{lemma4.7.1} and \eqref{lemma4.7.2}.

It is well-known that \cite{Nirenberg2011}
	\begin{align}
\label{2501241550}
&\|f(x_{\mm{h}},x_3)\|_{L^4(\mathbb{R}^2)}^2\lesssim \|f(x_{\mm{h}},x_3)\|_{L^2(\mathbb{R}^2)}\|\nabla_{\mm{h}}f(x_{\mm{h}},x_3)\|_{L^2(\mathbb{R}^2)}
\mbox{ for a.e. }x_3\in I_h,\\ 
&\label{2501s241550}
\|f(x_{\mm{h}},x_3)\|_{L^\infty(I_h)}^2\lesssim \|f(x_{\mm{h}},x_3)\|_{L^2(I_h)}
\|f(x_{\mm{h}},x_3)\|_{W^{1,2}(I_h)}
\mbox{ for a.e.  }x_{\mm{h}}\in \mathbb{R}^2 \end{align} 
and 
	\begin{align}
\label{2501241553} \|f\|_{L^\infty}^4 \lesssim\| f\|_{L^4} \|\nabla f\|_{L^4}^{3}. \end{align} 
 
Exploiting \eqref{2501241550} and \eqref{2501s241550}, we have
\begin{align}\|f \|_{L^4}^4\lesssim & \| \| f(x_{\mm{h}},x_3)\|_{L^2(\mathbb{R}^2)}^2\|\nabla_{\mm{h}}f(x_{\mm{h}},x_3)\|_{L^2(\mathbb{R}^2)}^2
\|_{L^1(I_h)}\nonumber \\
\lesssim & \|\| f(x_{\mm{h}},x_3)\|_{L^\infty(I_h)} \|_{L^2(\mathbb{R}^2)}^2\|\nabla_{\mm{h}}f \|_{0}^2  
\lesssim    \| f \|_{1}^2\|\nabla_{\mm{h}}f \|_{0}^2 .
\end{align}
Hence \eqref{lemma4.7.1} holds.
  Thus we further deduce from \eqref{lemma4.7.1} and \eqref{2501241553} that
\begin{align}\|f \|_{L^\infty}^2\lesssim &  \|  f\|_{W^{1,4}}^2\lesssim \|\nabla_{\mm{h}} f\|_{1}\| f\|_{2},\end{align}
which yields \eqref{lemma4.7.2}. \hfill $\Box$
\end{pf}

\begin{lem}\label{A.3}
Product estimates (see \cite[Lemma A.3]{JFJHJS2023}): Let $D\subset\mathbb{R}^3$ be a domain satisfying the cone condition, and
$f$, $g$ be functions defined in $D$. Then
\begin{align}\label{4}
\|fg\|_{i}\lesssim
\begin{cases}
\|f\|_{1}\|g\|_{1}&\mbox{for }i=0;\\
\|f\|_{i}\|g\|_{2}&\mbox{for }0\leqslant i\leqslant 2.
\end{cases}
\end{align}
if the norms on the right hand side of the above inequalities are finite.
\end{lem}

\begin{lem}\label{A.5}
A Hodge-type elliptic estimate \cite[Lemma A.4]{JJZ2022}: Let $i\geqslant1$, then
\begin{align}\label{7}
\|\nabla w\|_{i-1}\lesssim\|(\mm{div}  w,\mm{curl} w)\|_{i-1}\mbox{ for any }w\in H^i\mbox{ with }w_3|_{\partial\Omega}=0.
\end{align}
\end{lem}

\begin{lem}\label{A.8} Helmholtz decomposition in $L^2$-spaces (see Lemma 2.5.1 in Chapter II in \cite{MR1928881}): Let $D\subseteq \mathbb{R}^n$, $n\geqslant 2$, be any domain. We define that
$$\mathcal{L}^2(D):=\overline{\mathcal{C}_{0}^{\infty}(D )}^{\|\cdot\|_{L^2(D)}},\ \mathcal{C}_{0}^{\infty}(D ):=\{w\in C_{0}^{\infty}(D )~|~\mm{div}  w=0\} $$
and
$$G(D ):=\{f\in L^2(D )~|~f=\nabla g\mbox{ for some scalar function } g \in L^2_{\mm{loc}}(D)\}.$$
Then
$$G(D )=\left\{f\in L^2(D )~\bigg|~\int_D f\cdot w\mm{d}x=0\mbox{ for all }w\in \mathcal{L}^2(D )\right\},$$ 
and each $f\in L^2(D )$ has a unique decomposition
$$f=\tilde{f}+\nabla g$$
with $\tilde{f}\in \mathcal{L}^2(D )$, $\nabla g\in G(D )$, $\int_D \tilde{f}\cdot \nabla g\mm{d}x=0$ and
$$\|f\|_0^2=\|\tilde{f}\|_0^2+\|\nabla g\|_0^2.$$
\end{lem}

\begin{lem}\label{012243}
An elliptic estimate for the Dirichlet boundary value condition: Let $i\geqslant0$, ${f}^1\in H^{i}$ and ${f}^2\in H^{i+2}$ be given, then there exists a unique solution $w\in H^{i+2}$ solving the boundary value problem:
\begin{equation*}
\begin{cases}
\Delta w={f}^1&\mbox{in }\Omega,\\
w={f}^2&\mbox{on }\partial \Omega.
\end{cases}
\end{equation*}
Moreover,
\begin{equation}\label{11}
\|w\|_{i+2}\lesssim\|{f}^1\|_{i}+\|{f}^2\|_{i+2}.
\end{equation}
\end{lem}
\begin{pf}
Please refer to \cite[Lemma A.7]{JJZ2022} for the proof.
\hfill$\Box$
\end{pf}
\begin{lem}\label{A.9}
An elliptic estimate for the Neumann boundary value condition: Let $a$ be a positive constant and $i\geqslant 0$, then there exists a unique solution $w\in H^{i+2}$ solving the boundary value problem:
\begin{equation*}
\begin{cases}
-a\Delta w=\mm{div}{f}&\mbox{in }\Omega,\\
\partial_{{\mathbf{n}}}w=0&\mbox{on }\partial \Omega,
\end{cases}
\end{equation*}
where ${\mathbf{n}}$ denotes the outward unit normal vector to $\partial \Omega$. Moreover,
\begin{equation}\label{12}
\|\nabla w\|_{i+1}\lesssim\|{f}\|_{0}+\|\mm{div}{f}\|_{i}.
\end{equation}
\end{lem}
\begin{pf}
Please refer to \cite[Lemma A.7]{JJZ2022} and \cite[1.3.1 Theorem in Chapter III]{MR1928881}.
\hfill$\Box$
\end{pf}

\begin{lem}\label{lem2.1}
Let $(\sigma,u)\in  {\mathfrak{P}} _{T}\times {\mathcal{V}}_{T}$ be the solution of \eqref{1.7} with the initial condition $\sigma^0\in { {H}}^3_{\bar{\rho}}$, then $\sigma$ satisfies the following boundary condition:  
\begin{align}&\label{2.1}
(\sigma,\partial_3\rho, \partial_3^2\sigma)|_{\partial\Omega}=0, 
\end{align}
where $\rho=\bar{\rho}+\kappa  ^{-1 }\sigma$.
\end{lem}
\begin{proof}
The above result  \eqref{2.1} can be found in \cite[Lemma 2.1]{JZZa2024}. However we provide the proofs for reader's convenience. In view of the mass equation \eqref{1.7}$_1$ and the boundary condition of $u_3$ in \eqref{20220202081737}, it holds that
\begin{align*}
(\sigma_t+u_{\mm{h}}\cdot\nabla_{\mm{h}}\sigma)|_{\partial\Omega}=0.
\end{align*}
Taking the inner product of the above identity and $\sigma$ in $L^2{(\partial\Omega)}$, and then using the integration by parts and the embedding inequality of $H^2\hookrightarrow C^0(\overline{\Omega})$ in (\ref{1}), we derive that
$$\frac{\mm{d}}{\mm{d}t}\int_{\partial\Omega}|\sigma|^2\mm{d}x_{\mm{h}}
=-\int_{\partial\Omega}u_{\mm{h}}\cdot\nabla_{\mm{h}}|\sigma|^2\mm{d}x_{\mm{h}}
\lesssim\|\mm{div}_{\mm{h}} u_{\mm{h}}\|_2\int_{\partial\Omega}|\sigma|^2\mm{d}x_{\mm{h}}.$$
Noting that $\int_0^T\|\mm{div}_{\mm{h}} u_{\mm{h}}\|_2\mm{d}\tau<\infty$ and $\sigma^0|_{\partial\Omega}=0$, thus applying Gronwall's inequality to the above inequality yields
$$\|\sigma\|^2_{L^2(\partial\Omega)}=0,$$
which implies
\begin{align}\label{2.5}
\sigma|_{\partial\Omega}=0.
\end{align}	

Applying $\partial_3$ to \eqref{1.2}$_1$, then using the boundary condition   \eqref{20220202081737}, we can compute out that
\begin{equation*}
(\partial_{3}\rho_t+u_{\mm{h}}\cdot\nabla_{\mm{h}}\partial_{3}\rho
+\partial_{3}u_3\partial_{3}\rho)|_{\partial\Omega}=0,
\end{equation*}
where $\rho:=\bar{\rho}+\kappa  ^{-1 }\sigma$.
Following the similar argument of \eqref{2.5} by further using the incompressible condition \eqref{1.7}$_3$ and the boundary condition $\partial_3\rho^0|_{\partial\Omega}=0$, we easily derive from the above identity that
\begin{equation}\label{2.6}
\partial_3\rho|_{\partial\Omega}=0.
\end{equation}

Similarly, applying $\partial^2_3$ to \eqref{1.7}$_1$, and then using $\partial_3^2\sigma^0|_{\partial\Omega}=0$, the incompressible condition and the boundary conditions in \eqref{20220202081737}, \eqref{2.5}, we have
\begin{equation*} 	
(\partial^2_{3}\sigma_t+u_{\mm{h}}\cdot\nabla_{\mm{h}}\partial^2_{3}\sigma
+2\partial_{3}u_3\partial^2_{3}\sigma)|_{\partial\Omega}=0,
\end{equation*}
which obviously implies that
\begin{equation}\label{2.7}
\partial_3^2\sigma|_{\partial\Omega}=0.
\end{equation}
Thus we arrive at \eqref{2.1}, which presents that the solution $\sigma$ satisfies the same boundary conditions as well as the initial data $\sigma^0$. 
\end{proof}
\begin{lem}\label{lem2.6}
Under the assumption of Lemma \ref{lem2.1}, we have the following estimates:
	\begin{align} 
		&\| \nabla^iu\|_2\lesssim \|  \nabla^i u\|_0 +\| \Delta  \nabla^i u\|_0\mbox{ for }i=0,\ 1, \label{20240129182211}  \\
	&\|\sigma\|_3\lesssim\|\nabla \Delta \sigma\|_0,\label{2.52}\\
&\|\partial_s\sigma\|_2\lesssim\|\partial_s\Delta \sigma\|_0\mbox{ for }s=1,\ 2,\ t .\label{2.46}
\end{align} 
\end{lem}
\begin{proof}Similar results can be found in \cite[Lemma 2.2]{JZZa2024}, however we also provide the proofs of \eqref{20240129182211}--\eqref{2.46} for reader's convenience.  

Recalling the boundary condition of $u$ in  \eqref{20220202081737}, we use both the elliptic estimates in Lemmas \ref{012243} and \ref{A.9} to deduce that
$$\| \nabla u\|_1\lesssim\|\nabla u\|_0+ \| \Delta u\|_0$$ 
and
$$\| \nabla\partial_j u\|_1\lesssim\|\nabla \partial_j  u\|_0+ \| \Delta\partial_j  u\|_0\mbox{ for }1\leqslant j\leqslant 3,$$ 
which, together with the interpolation inequality \eqref{8},   imply \eqref{20240129182211}.

Similarly, thanks to the boundary conditions of $(\sigma,\partial_3^2\sigma)$ in \eqref{2.1}, we also exploit  the elliptic estimate  in Lemma \ref{012243} to deduce \eqref{2.52}--\eqref{2.46}. This completes the proof of Lemma \ref{lem2.6}. 
\end{proof}
\begin{thm}\label{A.7}
Aubin--Lions theorem \cite[Theorem 1.71]{NASII04}: Let $T>0$, $X\hookrightarrow\hookrightarrow B\hookrightarrow Y$ be Banach spaces, $\{f_n\}_{n=1}^\infty$ a sequence bounded in $L^q(I_T,B)\cap L^1(I_T,X)$ and $\{\mm{d}f_n/\mm{d}t\}_{n=1}^\infty$ bounded in $L^1(I_T,Y)$, where $1<q\leqslant \infty$. Then $\{f_n\}_{n=1}^\infty$ is relatively compact in $L^p(I_T,B)$ for any $1\leqslant p<q$.
\end{thm}

\begin{thm}\label{A.10}
Arzel\'a--Ascoli theorem \cite[Theorem 1.70]{NASII04}:
Let $T>0$ and $B$, $X$ be Banach spaces such that $B\hookrightarrow\hookrightarrow X$ is compact. Let $f_n$ be a sequence of functions $\overline{I_T}\rightarrow B$ uniformly bounded in $B$ and uniformly continuous in $X$. Then there exists $f\in C^0(\overline{I_T},B)$ such that $f_n\rightarrow f$ strongly in $C^0(\overline{I_T},X)$ at least for a chosen subsequence.
\end{thm}\begin{thm} \label{asdA.7}
 Banach--Alaoglu theorem (see Sections 1.4.5.25 and 1.4.5.26 in \cite{NASII04}):
  \begin{enumerate}[(1)]
\item  Let $X$ be a reflexive Banach space and let $\{u_n\}\subset X$ be a bounded sequence. Then there exists a subsequence $\{u_{n_k}\}_{k=1}^{\infty}$ weakly convergent in $X$. 
\item Another version of the Banach--Alaoglu theorem: Let $X$ be a separable Banach space and let $\{f_n\}\subset X^*$ be a bounded sequence. Then there exists a subsequence $\{f_{n_k}\}_{k=1}^{\infty}$ weakly-$*$ convergent in $X^*$.
 \end{enumerate} 
\end{thm}

\vspace{4mm}
\noindent\textbf{Acknowledgements.}
The research of Fei Jiang was supported by NSFC (Grant Nos. 12371233 and 12231016), and the Natural Science Foundation of Fujian Province of China (Grant Nos. 2024J011011 and 2022J01105) and the Central Guidance on Local Science and Technology Development Fund of Fujian Province (Grant No. 2023L3003).

\renewcommand\refname{References}
\renewenvironment{thebibliography}[1]{%
\section*{\refname}
\list{{\arabic{enumi}}}{\def\makelabel##1{\hss{##1}}\topsep=0mm
\parsep=0mm
\partopsep=0mm\itemsep=0mm
\labelsep=1ex\itemindent=0mm
\settowidth\labelwidth{\small[#1]}
\leftmargin\labelwidth \advance\leftmargin\labelsep
\advance\leftmargin -\itemindent
\usecounter{enumi}}\small
\def\newblock{\ }
\sloppy\clubpenalty4000\widowpenalty4000
\sfcode`\.=1000\relax}{\endlist}
\bibliographystyle{model1b-num-names}
\bibliography{refs}

\begin{thebibliography}{62}
\expandafter\ifx\csname natexlab\endcsname\relax\def\natexlab#1{#1}\fi
\providecommand{\bibinfo}[2]{#2}
\ifx\xfnm\relax \def\xfnm[#1]{\unskip,\space#1}\fi
\bibitem[{Adams and Fourier(2003)}]{Adams1975}
\bibinfo{author}{R.A. Adams}, \bibinfo{author}{J.J.F. Fourier},
  \bibinfo{title}{{Sobolev Spaces}}, \bibinfo{publisher}{Academic press, New
  York}, \bibinfo{year}{2003}.
\bibitem[{Antonelli and Spirito(2022)}]{MR4412067}
\bibinfo{author}{P.~Antonelli}, \bibinfo{author}{S.~Spirito},
  \bibinfo{title}{Global existence of weak solutions to the
  {N}avier--{S}tokes--{K}orteweg equations}, \bibinfo{journal}{Ann. Inst. H.
  Poincar\'e C Anal. Non Lin\'eaire} \bibinfo{volume}{39}
  (\bibinfo{year}{2022}) \bibinfo{pages}{171--200}.
\bibitem[{Audiard and Haspot(2017)}]{MR3613503}
\bibinfo{author}{C.~Audiard}, \bibinfo{author}{B.~Haspot},
  \bibinfo{title}{Global well-posedness of the {E}uler--{K}orteweg system for
  small irrotational data}, \bibinfo{journal}{Comm. Math. Phys.}
  \bibinfo{volume}{351} (\bibinfo{year}{2017}) \bibinfo{pages}{201--247}.
\bibitem[{Benzoni-Gavage et~al.(2007)Benzoni-Gavage, Danchin and
  Descombes}]{MR2354691}
\bibinfo{author}{S.~Benzoni-Gavage}, \bibinfo{author}{R.~Danchin},
  \bibinfo{author}{S.~Descombes}, \bibinfo{title}{On the well-posedness for the
  {E}uler--{K}orteweg model in several space dimensions},
  \bibinfo{journal}{Indiana Univ. Math. J.} \bibinfo{volume}{56}
  (\bibinfo{year}{2007}) \bibinfo{pages}{1499--1579}.
\bibitem[{Bian et~al.(2014)Bian, Yao and Zhu}]{BYZ2014}
\bibinfo{author}{D.~Bian}, \bibinfo{author}{L.~Yao}, \bibinfo{author}{C.~Zhu},
  \bibinfo{title}{{Vanishing capillarity limit of the compressible fluid models
  of Korteweg type to the Navier--Stokes equations}}, \bibinfo{journal}{SIAM J.
  Math. Anal.} \bibinfo{volume}{46} (\bibinfo{year}{2014})
  \bibinfo{pages}{1633--1650}.
\bibitem[{Bresch et~al.(2008)Bresch, Desjardins, Gisclon and Sart}]{BDG2008}
\bibinfo{author}{D.~Bresch}, \bibinfo{author}{B.~Desjardins},
  \bibinfo{author}{M.~Gisclon}, \bibinfo{author}{R.~Sart},
  \bibinfo{title}{{Instability results related to compressible Korteweg
  system}}, \bibinfo{journal}{Ann. Univ. Ferrara Sez.} \bibinfo{volume}{54}
  (\bibinfo{year}{2008}) \bibinfo{pages}{11--36}.
\bibitem[{Bresch et~al.(2003)Bresch, Desjardins and Lin}]{DDL2003}
\bibinfo{author}{D.~Bresch}, \bibinfo{author}{B.~Desjardins},
  \bibinfo{author}{C.K. Lin}, \bibinfo{title}{On some compressible fluid
  models: {K}orteweg, lubrication, and shallow water systems},
  \bibinfo{journal}{Comm. Partial Differential Equations} \bibinfo{volume}{28}
  (\bibinfo{year}{2003}) \bibinfo{pages}{843--868}.
\bibitem[{Bresch et~al.(2019)Bresch, Gisclon and Lacroix-Violet}]{MR3961293}
\bibinfo{author}{D.~Bresch}, \bibinfo{author}{M.~Gisclon},
  \bibinfo{author}{I.~Lacroix-Violet}, \bibinfo{title}{On
  {N}avier--{S}tokes--{K}orteweg and {E}uler--{K}orteweg systems: application
  to quantum fluids models}, \bibinfo{journal}{Arch. Ration. Mech. Anal.}
  \bibinfo{volume}{233} (\bibinfo{year}{2019}) \bibinfo{pages}{975--1025}.
\bibitem[{Bresch et~al.(2022)Bresch, Gisclon, Lacroix-Violet and
  Vasseur}]{MR4354130}
\bibinfo{author}{D.~Bresch}, \bibinfo{author}{M.~Gisclon},
  \bibinfo{author}{I.~Lacroix-Violet}, \bibinfo{author}{A.~Vasseur},
  \bibinfo{title}{On the exponential decay for compressible
  {N}avier--{S}tokes--{K}orteweg equations with a drag term},
  \bibinfo{journal}{J. Math. Fluid Mech.} \bibinfo{volume}{24}
  (\bibinfo{year}{2022}) \bibinfo{pages}{Paper No. 11, 16}.
\bibitem[{Browning and Kreiss(1982)}]{Browning1982}
\bibinfo{author}{G.~Browning}, \bibinfo{author}{H.O. Kreiss},
  \bibinfo{title}{{Problems with different time scales for nonlinear partial
  differential equations}}, \bibinfo{journal}{SIAM J. Appl. Math.}
  \bibinfo{volume}{42} (\bibinfo{year}{1982}) \bibinfo{pages}{704--718}.
\bibitem[{Burtea and Charve(2017)}]{BC2017}
\bibinfo{author}{C.~Burtea}, \bibinfo{author}{F.~Charve},
  \bibinfo{title}{{Lagrangian methods for a general inhomogeneous
  incompressible Navier--Stokes--Korteweg system with variable capillarity and
  viscosity coefficients}}, \bibinfo{journal}{SIAM J. Math. Anal.}
  \bibinfo{volume}{49} (\bibinfo{year}{2017}) \bibinfo{pages}{3476--3495}.
\bibitem[{Charve(2014)}]{MR3160440}
\bibinfo{author}{F.~Charve}, \bibinfo{title}{Local in time results for local
  and non-local capillary {N}avier--{S}tokes systems with large data},
  \bibinfo{journal}{J. Differential Equations} \bibinfo{volume}{256}
  (\bibinfo{year}{2014}) \bibinfo{pages}{2152--2193}.
\bibitem[{Charve et~al.(2021)Charve, Danchin and Xu}]{MR4340485}
\bibinfo{author}{F.~Charve}, \bibinfo{author}{R.~Danchin},
  \bibinfo{author}{J.~Xu}, \bibinfo{title}{Gevrey analyticity and decay for the
  compressible {N}avier--{S}tokes system with capillarity},
  \bibinfo{journal}{Indiana Univ. Math. J.} \bibinfo{volume}{70}
  (\bibinfo{year}{2021}) \bibinfo{pages}{1903--1944}.
\bibitem[{Charve and Haspot(2013)}]{MR3032985}
\bibinfo{author}{F.~Charve}, \bibinfo{author}{B.~Haspot},
  \bibinfo{title}{Existence of a global strong solution and vanishing
  capillarity-viscosity limit in one dimension for the {K}orteweg system},
  \bibinfo{journal}{SIAM J. Math. Anal.} \bibinfo{volume}{45}
  (\bibinfo{year}{2013}) \bibinfo{pages}{469--494}.
\bibitem[{Chen and Li(2021)}]{MR4228314}
\bibinfo{author}{Z.~Chen}, \bibinfo{author}{Y.~Li}, \bibinfo{title}{Asymptotic
  behavior of solutions to an impermeable wall problem of the compressible
  fluid models of {K}orteweg type with density-dependent viscosity and
  capillarity}, \bibinfo{journal}{SIAM J. Math. Anal.} \bibinfo{volume}{53}
  (\bibinfo{year}{2021}) \bibinfo{pages}{1434--1473}.
\bibitem[{Chen and Zhao(2014)}]{MR3168914}
\bibinfo{author}{Z.~Chen}, \bibinfo{author}{H.~Zhao}, \bibinfo{title}{Existence
  and nonlinear stability of stationary solutions to the full compressible
  {N}avier--{S}tokes--{K}orteweg system}, \bibinfo{journal}{J. Math. Pures
  Appl.} \bibinfo{volume}{101} (\bibinfo{year}{2014})
  \bibinfo{pages}{330--371}.
\bibitem[{Danchin and Desjardins(2001)}]{MR1810272}
\bibinfo{author}{R.~Danchin}, \bibinfo{author}{B.~Desjardins},
  \bibinfo{title}{Existence of solutions for compressible fluid models of
  {K}orteweg type}, \bibinfo{journal}{Ann. Inst. H. Poincar\'e C Anal. Non
  Lin\'eaire} \bibinfo{volume}{18} (\bibinfo{year}{2001})
  \bibinfo{pages}{97--133}.
\bibitem[{Ding et~al.(2018)Ding, Li and Xin}]{DLX2018}
\bibinfo{author}{S.~Ding}, \bibinfo{author}{Q.~Li}, \bibinfo{author}{Z.~Xin},
  \bibinfo{title}{{Stability analysis for the incompressible Navier--Stokes
  equations with Navier boundary conditions}}, \bibinfo{journal}{J. Math. Fluid
  Mech.} \bibinfo{volume}{20} (\bibinfo{year}{2018}) \bibinfo{pages}{603--629}.
\bibitem[{Dunn and Serrin(1985)}]{DS1985}
\bibinfo{author}{J.E. Dunn}, \bibinfo{author}{J.~Serrin}, \bibinfo{title}{{On
  the thermomechanics of interstitial working}}, \bibinfo{journal}{Arch.
  Ration. Mech. Anal.} \bibinfo{volume}{88} (\bibinfo{year}{1985})
  \bibinfo{pages}{95--133}.
\bibitem[{Fanelli(2016)}]{F2016}
\bibinfo{author}{F.~Fanelli}, \bibinfo{title}{{Highly rotating viscous
  compressible fluids in presence of capillarity effects}},
  \bibinfo{journal}{Math. Ann.} \bibinfo{volume}{366} (\bibinfo{year}{2016})
  \bibinfo{pages}{981--1033}.
\bibitem[{Germain and LeFloch(2016)}]{MR3433629}
\bibinfo{author}{P.~Germain}, \bibinfo{author}{P.~LeFloch},
  \bibinfo{title}{Finite energy method for compressible fluids: the
  {N}avier--{S}tokes--{K}orteweg model}, \bibinfo{journal}{Comm. Pure Appl.
  Math.} \bibinfo{volume}{69} (\bibinfo{year}{2016}) \bibinfo{pages}{3--61}.
\bibitem[{Goto(1990)}]{Goto1990}
\bibinfo{author}{S.~Goto}, \bibinfo{title}{{Singular limit of the
  incompressible ideal magneto-fluid motion with respect to the Alfv{\'e}n
  number}}, \bibinfo{journal}{Hokkaido Math. J.} \bibinfo{volume}{19}
  (\bibinfo{year}{1990}) \bibinfo{pages}{175--187}.
\bibitem[{Guo and Tice(2013)}]{GT2013}
\bibinfo{author}{Y.~Guo}, \bibinfo{author}{I.~Tice}, \bibinfo{title}{{Decay of
  viscous surface waves without surface tension in horizontally
  infinitedomains}}, \bibinfo{journal}{Anal. PDE} \bibinfo{volume}{6}
  (\bibinfo{year}{2013}) \bibinfo{pages}{1429--1533}.
\bibitem[{Haspot(2011)}]{MR2805863}
\bibinfo{author}{B.~Haspot}, \bibinfo{title}{Existence of global weak solution
  for compressible fluid models of {K}orteweg type}, \bibinfo{journal}{J. Math.
  Fluid Mech.} \bibinfo{volume}{13} (\bibinfo{year}{2011})
  \bibinfo{pages}{223--249}.
\bibitem[{Haspot(2020)}]{MR4201123}
\bibinfo{author}{B.~Haspot}, \bibinfo{title}{Strong solution for {K}orteweg
  system in {$\rm{BMO}^{-1} (\Bbb R^N)$} with initial density in {$L^\infty$}},
  \bibinfo{journal}{Proc. Lond. Math. Soc. (3)} \bibinfo{volume}{121}
  (\bibinfo{year}{2020}) \bibinfo{pages}{1766--1797}.
\bibitem[{Hattori and Li(1994)}]{HL1994}
\bibinfo{author}{H.~Hattori}, \bibinfo{author}{D.~Li},
  \bibinfo{title}{{Solutions for two-dimensional system for materials of
  Korteweg type}}, \bibinfo{journal}{SIAM J. Math. Anal.} \bibinfo{volume}{25}
  (\bibinfo{year}{1994}) \bibinfo{pages}{85--98}.
\bibitem[{Hong(2020)}]{MR4099657}
\bibinfo{author}{H.~Hong}, \bibinfo{title}{Strong solutions for the
  compressible barotropic fluid model of {K}orteweg type in the bounded
  domain}, \bibinfo{journal}{Z. Angew. Math. Phys.} \bibinfo{volume}{71}
  (\bibinfo{year}{2020}) \bibinfo{pages}{Paper No. 85, 25}.
\bibitem[{Hong(2022)}]{MR4369180}
\bibinfo{author}{H.~Hong}, \bibinfo{title}{Stability of stationary solutions
  and viscous shock wave in the inflow problem for isentropic
  {N}avier--{S}tokes--{K}orteweg system}, \bibinfo{journal}{J. Differential
  Equations} \bibinfo{volume}{314} (\bibinfo{year}{2022})
  \bibinfo{pages}{518--573}.
\bibitem[{Hou et~al.(2018)Hou, Peng and Zhu}]{HPZ2018NA}
\bibinfo{author}{X.~Hou}, \bibinfo{author}{H.~Peng}, \bibinfo{author}{C.~Zhu},
  \bibinfo{title}{{Global well-posedness of the 3D non-isothermal compressible
  fluid model of Korteweg type}}, \bibinfo{journal}{Nonlinear Anal. Real World
  Applications} \bibinfo{volume}{43} (\bibinfo{year}{2018})
  \bibinfo{pages}{18--53}.
\bibitem[{Hou et~al.(2017)Hou, Yao and Zhu}]{HYZ2017}
\bibinfo{author}{X.~Hou}, \bibinfo{author}{L.~Yao}, \bibinfo{author}{C.~Zhu},
  \bibinfo{title}{{Vanishing capillarity limit of the compressible
  non-isentropic Navier--Stokes--Korteweg system to Navier--Stokes system}},
  \bibinfo{journal}{J. Math. Anal. Appl.} \bibinfo{volume}{448}
  (\bibinfo{year}{2017}) \bibinfo{pages}{421--446}.
\bibitem[{Jiang et~al.(2023{\natexlab{a}})Jiang, Jiang and Jiang}]{JFJHJS2023}
\bibinfo{author}{F.~Jiang}, \bibinfo{author}{H.~Jiang},
  \bibinfo{author}{S.~Jiang}, \bibinfo{title}{{Rayleigh--Taylor instability in
  stratified compressible fluids with/without the interfacial surface
  tension}}, \bibinfo{journal}{ariXiv:2309.13370}
  (\bibinfo{year}{2023}{\natexlab{a}}).
\bibitem[{Jiang et~al.(2022)Jiang, Jiang and Zhao}]{JJZ2022}
\bibinfo{author}{F.~Jiang}, \bibinfo{author}{S.~Jiang},
  \bibinfo{author}{Y.~Zhao}, \bibinfo{title}{{On inhibition of the
  Rayleigh--Taylor instability by a horizontal magnetic field in ideal MHD
  fluids with velocity damping}}, \bibinfo{journal}{J. Differential Equations}
  \bibinfo{volume}{314} (\bibinfo{year}{2022}) \bibinfo{pages}{574--652}.
\bibitem[{Jiang et~al.(2023{\natexlab{b}})Jiang, Jiang and Zhao}]{MR4587686}
\bibinfo{author}{F.~Jiang}, \bibinfo{author}{S.~Jiang},
  \bibinfo{author}{Y.~Zhao}, \bibinfo{title}{On inhibition of the
  {R}ayleigh--{T}aylor instability by a horizontal magnetic field in 2{D}
  non-resistive {MHD} fluids: the viscous case}, \bibinfo{journal}{CSIAM Trans.
  Appl. Math.} \bibinfo{volume}{4} (\bibinfo{year}{2023}{\natexlab{b}})
  \bibinfo{pages}{451--514}.
\bibitem[{Jiang et~al.(2023{\natexlab{c}})Jiang, Li and Zhang}]{JLZa2023}
\bibinfo{author}{F.~Jiang}, \bibinfo{author}{F.~Li},
  \bibinfo{author}{Z.~Zhang}, \bibinfo{title}{{On stability and instability of
  gravity driven Navier--Stokes--Korteweg model in two dimensions}},
  \bibinfo{journal}{arXiv preprint arXiv:2302.01013}
  (\bibinfo{year}{2023}{\natexlab{c}}).
\bibitem[{Jiang et~al.(2024)Jiang, Zhang and Zhang}]{JZZa2024}
\bibinfo{author}{F.~Jiang}, \bibinfo{author}{Y.~Zhang},
  \bibinfo{author}{Z.~Zhang}, \bibinfo{title}{{On the inhibition of
  Rayleigh--Taylor instability by capillarity in the Navier--Stokes--Korteweg
  model}}, \bibinfo{journal}{arXiv preprint arXiv:2402.07201}
  (\bibinfo{year}{2024}).
\bibitem[{Jiang et~al.(2019)Jiang, Ju and Xu}]{JJX2019}
\bibinfo{author}{S.~Jiang}, \bibinfo{author}{Q.~Ju}, \bibinfo{author}{X.~Xu},
  \bibinfo{title}{{Small Alfv{\'e}n number limit for incompressible
  magneto-hydrodynamics in a domain with boundaries}}, \bibinfo{journal}{Sci.
  China Math.} \bibinfo{volume}{62} (\bibinfo{year}{2019})
  \bibinfo{pages}{2229--2248}.
\bibitem[{Ju and Xu(2022)}]{JX2022}
\bibinfo{author}{Q.~Ju}, \bibinfo{author}{J.~Xu}, \bibinfo{title}{{Zero-Mach
  limit of the compressible Navier--Stokes--Korteweg equations}},
  \bibinfo{journal}{J. Math. Phys.} \bibinfo{volume}{63} (\bibinfo{year}{2022})
  \bibinfo{pages}{111503}.
\bibitem[{J\"ungel et~al.(2014)J\"ungel, Lin and Wu}]{MR3210149}
\bibinfo{author}{A.~J\"ungel}, \bibinfo{author}{C.K. Lin},
  \bibinfo{author}{K.C. Wu}, \bibinfo{title}{An asymptotic limit of a
  {N}avier--{S}tokes system with capillary effects}, \bibinfo{journal}{Comm.
  Math. Phys.} \bibinfo{volume}{329} (\bibinfo{year}{2014})
  \bibinfo{pages}{725--744}.
\bibitem[{Kawashima et~al.(2021)Kawashima, Shibata and Xu}]{MR4312286}
\bibinfo{author}{S.~Kawashima}, \bibinfo{author}{Y.~Shibata},
  \bibinfo{author}{J.~Xu}, \bibinfo{title}{The {$L^p$} energy methods and decay
  for the compressible {N}avier--{S}tokes equations with capillarity},
  \bibinfo{journal}{J. Math. Pures Appl. (9)} \bibinfo{volume}{154}
  (\bibinfo{year}{2021}) \bibinfo{pages}{146--184}.
\bibitem[{Kobayashi et~al.(2022)Kobayashi, Murata and Saito}]{MR4354131}
\bibinfo{author}{T.~Kobayashi}, \bibinfo{author}{M.~Murata},
  \bibinfo{author}{H.~Saito}, \bibinfo{title}{Resolvent estimates for a
  compressible fluid model of {K}orteweg type and their application},
  \bibinfo{journal}{J. Math. Fluid Mech.} \bibinfo{volume}{24}
  (\bibinfo{year}{2022}) \bibinfo{pages}{Paper No. 12, 42}.
\bibitem[{Korteweg(1901)}]{korteweg1901forme}
\bibinfo{author}{D.J. Korteweg}, \bibinfo{title}{{Sur la forme que prennent les
  {\'e}quations du mouvements des fluides si l'on tient compte des forces
  capillaires caus{\'e}es par des variations de densit{\'e} consid{\'e}rables
  mais connues et sur la th{\'e}orie de la capillarit{\'e} dans l'hypoth{\`e}se
  d'une variation continue de la densit{\'e}}}, \bibinfo{journal}{Archives
  N{\'e}erlandaises des Sciences exactes et naturelles} \bibinfo{volume}{6}
  (\bibinfo{year}{1901}) \bibinfo{pages}{1--24}.
\bibitem[{Kotschote(2008)}]{MR2436788}
\bibinfo{author}{M.~Kotschote}, \bibinfo{title}{Strong solutions for a
  compressible fluid model of {K}orteweg type}, \bibinfo{journal}{Ann. Inst. H.
  Poincar\'e C Anal. Non Lin\'eaire} \bibinfo{volume}{25}
  (\bibinfo{year}{2008}) \bibinfo{pages}{679--696}.
\bibitem[{Kotschote(2010)}]{MR2749439}
\bibinfo{author}{M.~Kotschote}, \bibinfo{title}{Strong well-posedness for a
  {K}orteweg-type model for the dynamics of a compressible non-isothermal
  fluid}, \bibinfo{journal}{J. Math. Fluid Mech.} \bibinfo{volume}{12}
  (\bibinfo{year}{2010}) \bibinfo{pages}{473--484}.
\bibitem[{Kotschote(2012)}]{K2012}
\bibinfo{author}{M.~Kotschote}, \bibinfo{title}{{Dynamics of compressible
  non-isothermal fluids of non-Newtonian Korteweg type}},
  \bibinfo{journal}{SIAM J. Math. Anal.} \bibinfo{volume}{44}
  (\bibinfo{year}{2012}) \bibinfo{pages}{74--101}.
\bibitem[{Kotschote(2014)}]{K2014}
\bibinfo{author}{M.~Kotschote}, \bibinfo{title}{{Existence and time-asymptotics
  of global strong solutions to dynamic Korteweg models}},
  \bibinfo{journal}{Indiana Univ. Math. J.} \bibinfo{volume}{63}
  (\bibinfo{year}{2014}) \bibinfo{pages}{21--51}.
\bibitem[{Li and Zhang(2023)}]{MR4622412}
\bibinfo{author}{F.~Li}, \bibinfo{author}{Z.~Zhang},
  \bibinfo{title}{Stabilizing effect of capillarity in the {R}ayleigh--{T}aylor
  problem to the viscous incompressible capillary fluids},
  \bibinfo{journal}{SIAM J. Math. Anal.} \bibinfo{volume}{55}
  (\bibinfo{year}{2023}) \bibinfo{pages}{3287--3315}.
\bibitem[{Li and Ding(2021)}]{MR4342183}
\bibinfo{author}{Q.~Li}, \bibinfo{author}{S.~Ding}, \bibinfo{title}{Global
  well-posedness of the {N}avier--{S}tokes equations with {N}avier-slip
  boundary conditions in a strip domain}, \bibinfo{journal}{Commun. Pure Appl.
  Anal.} \bibinfo{volume}{20} (\bibinfo{year}{2021})
  \bibinfo{pages}{3561--3581}.
\bibitem[{Li and Yong(2016)}]{LY2016}
\bibinfo{author}{Y.~Li}, \bibinfo{author}{W.~Yong}, \bibinfo{title}{{Zero Mach
  number limit of the compressible Navier--Stokes--Korteweg equations}},
  \bibinfo{journal}{Commun. Math. Sci.} \bibinfo{volume}{14}
  (\bibinfo{year}{2016}) \bibinfo{pages}{233--247}.
\bibitem[{Lin(2012)}]{Lin2012SomeAI}
\bibinfo{author}{F.~Lin}, \bibinfo{title}{{Some analytical issues for elastic
  complex fluids}}, \bibinfo{journal}{Comm. Pure Appl. Math.}
  \bibinfo{volume}{65} (\bibinfo{year}{2012}) \bibinfo{pages}{893--919}.
\bibitem[{Murata and Shibata(2020)}]{MR4188843}
\bibinfo{author}{M.~Murata}, \bibinfo{author}{Y.~Shibata}, \bibinfo{title}{The
  global well-posedness for the compressible fluid model of {K}orteweg type},
  \bibinfo{journal}{SIAM J. Math. Anal.} \bibinfo{volume}{52}
  (\bibinfo{year}{2020}) \bibinfo{pages}{6313--6337}.
\bibitem[{Nguyen(2023)}]{Na2023}
\bibinfo{author}{T.T. Nguyen}, \bibinfo{title}{{Influence of capillary number
  on nonlinear Rayleigh--Taylor instability to the Navier--Stokes--Korteweg
  equations}}, \bibinfo{journal}{arXiv preprint arXiv:2312.05536}
  (\bibinfo{year}{2023}).
\bibitem[{Nirenberg(1959)}]{Nirenberg2011}
\bibinfo{author}{L.~Nirenberg}, \bibinfo{title}{{On elliptic partial
  differential equations}}, \bibinfo{journal}{Estratto Ann. Sc. Norm. Super.
  Pisa (3)} \bibinfo{volume}{XIII(II)} (\bibinfo{year}{1959})
  \bibinfo{pages}{1--48}.
\bibitem[{Novotn{\'y} and Stra{\v{s}}kraba(2004)}]{NASII04}
\bibinfo{author}{A.~Novotn{\'y}}, \bibinfo{author}{I.~Stra{\v{s}}kraba},
  \bibinfo{title}{{Introduction to the Mathematical Theory of Compressible
  Flow}}, \bibinfo{publisher}{Oxford University Press, Oxford},
  \bibinfo{year}{2004}.
\bibitem[{Saito(2020)}]{MR4047969}
\bibinfo{author}{H.~Saito}, \bibinfo{title}{On the maximal {$L_p$}--{$L_q$}
  regularity for a compressible fluid model of {K}orteweg type on general
  domains}, \bibinfo{journal}{J. Differential Equations} \bibinfo{volume}{268}
  (\bibinfo{year}{2020}) \bibinfo{pages}{2802--2851}.
\bibitem[{Sha and Li(2019)}]{SL2019}
\bibinfo{author}{K.~Sha}, \bibinfo{author}{Y.~Li}, \bibinfo{title}{{Low Mach
  number limit of the three-dimensional full compressible
  Navier--Stokes--Korteweg equations}}, \bibinfo{journal}{Z. Angew. Math.
  Phys.} \bibinfo{volume}{70} (\bibinfo{year}{2019}) \bibinfo{pages}{169--175}.
\bibitem[{Sohr(2001)}]{MR1928881}
\bibinfo{author}{H.~Sohr}, \bibinfo{title}{The {N}avier--{S}tokes equations:
  {A}n elementary functional analytic approach},
  \bibinfo{publisher}{Birkh\"auser Verlag, Basel}, \bibinfo{year}{2001}.
\bibitem[{Sy et~al.(2006)Sy, Bresch, Guill\'en-Gonz\'alez, Lemoine and
  Rodr\'iguez-Bellido}]{MR2198187}
\bibinfo{author}{M.~Sy}, \bibinfo{author}{D.~Bresch},
  \bibinfo{author}{F.~Guill\'en-Gonz\'alez}, \bibinfo{author}{J.~Lemoine},
  \bibinfo{author}{M.A. Rodr\'iguez-Bellido}, \bibinfo{title}{Local strong
  solution for the incompressible {K}orteweg model}, \bibinfo{journal}{C. R.
  Math. Acad. Sci. Paris} \bibinfo{volume}{342} (\bibinfo{year}{2006})
  \bibinfo{pages}{169--174}.
\bibitem[{Tan and Wang(2010)}]{TW2010}
\bibinfo{author}{Z.~Tan}, \bibinfo{author}{Y.~Wang}, \bibinfo{title}{{Strong
  solutions for the incompressible fluid models of Korteweg type}},
  \bibinfo{journal}{Acta Math. Sci. Ser. B (Engl. Ed.)} \bibinfo{volume}{30}
  (\bibinfo{year}{2010}) \bibinfo{pages}{799--809}.
\bibitem[{Wang and Zhang(2024)}]{MR4713202}
\bibinfo{author}{P.~Wang}, \bibinfo{author}{Z.~Zhang},
  \bibinfo{title}{{Vanishing capillarity-viscosity limit of the incompressible
  Navier--Stokes--Korteweg equations with slip boundary condition}},
  \bibinfo{journal}{Nonlinear Anal.} \bibinfo{volume}{243}
  (\bibinfo{year}{2024}) \bibinfo{pages}{113526}.
\bibitem[{Yang et~al.(2015)Yang, Yao and Zhu}]{YYZ2015}
\bibinfo{author}{J.~Yang}, \bibinfo{author}{L.~Yao}, \bibinfo{author}{C.~Zhu},
  \bibinfo{title}{{Vanishing capillarity-viscosity limit for the incompressible
  inhomogeneous fluid models of Korteweg type}}, \bibinfo{journal}{Z. Angew.
  Math. Phys.} \bibinfo{volume}{66} (\bibinfo{year}{2015})
  \bibinfo{pages}{2285--2303}.
\bibitem[{Zhang et~al.(2023)Zhang, Tian and Wang}]{MR4645122}
\bibinfo{author}{X.~Zhang}, \bibinfo{author}{F.~Tian},
  \bibinfo{author}{W.~Wang}, \bibinfo{title}{On {R}ayleigh--{T}aylor
  instability in {N}avier--{S}tokes--{K}orteweg equations},
  \bibinfo{journal}{J. Inequal. Appl.}  (\bibinfo{year}{2023})
  \bibinfo{pages}{Paper No. 119, 30}.
\bibitem[{Zhang(2022)}]{MR4434208}
\bibinfo{author}{Z.~Zhang}, \bibinfo{title}{Rayleigh--{T}aylor instability for
  viscous incompressible capillary fluids}, \bibinfo{journal}{J. Math. Fluid
  Mech.} \bibinfo{volume}{24} (\bibinfo{year}{2022}) \bibinfo{pages}{Paper No.
  70, 23}.

\end{thebibliography}
\end{CJK*}
\end{document}